\documentclass[10pt,A4, leqno]{amsart}
\date{June 10, 2024}
\usepackage{graphicx}
\usepackage[mathscr]{eucal}
\usepackage{mathrsfs}
\usepackage{amscd}
\usepackage{amsfonts}
\usepackage{amsmath,amsthm,amssymb}
\usepackage{latexsym}
\numberwithin{equation}{section}
\usepackage{amsthm}
\theoremstyle{plain}
 \newtheorem{theorem}{Theorem}[section]
 \newtheorem*{theorem*}{Theorem}

 \newtheorem*{lemma*}{Lemma}
 \newtheorem{proposition}[theorem]{Proposition}
 
 \newtheorem*{fact*}{Fact}

 \newtheorem{lemma}[theorem]{Lemma}
 \newtheorem{corollary}[theorem]{Corollary}
\theoremstyle{remark}
 \newtheorem{definition}[theorem]{Definition}
 \newtheorem{remark}[theorem]{Remark}
 \newtheorem*{remark*}{Remark}
 \newtheorem*{acknowledgements}{Acknowledgements}
 \newtheorem{example}[theorem]{Example}
\numberwithin{equation}{section}

\newcommand{\vect}[1]{\boldsymbol{#1}}

\newcommand{\R}{{\Bbb R}}

\newcommand{\Z}{{\Bbb Z}}

\renewcommand{\phi}{\varphi}
\renewcommand{\epsilon}{\varepsilon}

\newcommand{\op}{\operatorname}

\newcommand{\mc}[1]{{\mathcal #1}}
\newcommand{\mb}[1]{{\mathbf #1}}
\newcommand{\pmt}[1]{{\begin{pmatrix} #1  \end{pmatrix}}}

\newcommand{\inner}[2]{\left\langle{#1},{#2}\right\rangle}

\title[]{%
Null hypersurfaces as 
wave fronts in Lorentz-Minkowski space}

\author[]{
S.~Akamine, A.~Honda,  M.~Umehara and  
        K.~Yamada 
}   

\address[Shintaro Akamine]{%
Department of Liberal Arts, College of Bioresource
Sciences,
Nihon University,1866 Kameino, Fujisawa, Kanagawa, 252-0880, Japan
}
\email{akamine.shintaro@nihon-u.ac.jp}

\address[Atsufumi Honda]{
Department of Applied Mathematics, 
Faculty of Engineering, Yokohama National University,
79-5 Tokiwadai, Hodogaya, Yokohama 240-8501, Japan
}
\email{honda-atsufumi-kp@ynu.ac.jp}

\address[Masaaki Umehara]{%
   Department of Mathematical and Computing Sciences,
   Tokyo Institute of Technology,
   Tokyo 152-8552, Japan
}
\email{umehara@is.titech.ac.jp}

\address[Kotaro Yamada]{%
   Department of Mathematics,
   Tokyo Institute of Technology,
   Tokyo 152-8551, Japan
}
\email{kotaro@math.titech.ac.jp}


\thanks{
The first author was
partially supported by the
Grant-in-Aid for Young Scientists No.~19K14527.
The second author was
partially supported by the
Grant-in-Aid for Young Scientists 
No.~19K14526 and (B) No.\  20H01801.
The third and fourth authors were partially supported by
(B) No.\  21H00981 and (B) No.\  17H02839, respectively, from 
Japan Society for the Promotion of Science.
}

\usepackage[active]{srcltx}

\begin{document}
\maketitle

\begin{abstract}
In this paper, we show that \lq\lq $L$-complete
null hypersurfaces'' 
(i.e. ruled hypersurfaces foliated by entirety of light-like lines)
as wave fronts in the $(n+1)$-dimensional 
Lorentz-Minkowski space  
are canonically induced by 
hypersurfaces in the $n$-dimensional Euclidean space.
As an application, we show that
most of null wave fronts can be
realized as  restrictions of
certain $L$-complete
null wave fronts. 
Moreover, we determine $L$-complete
null wave fronts whose singular sets are compact.
\end{abstract}

\tableofcontents

\section*{Introduction} \label{sec:I} 

We denote by $\R^{n+1}_1$ the $(n+1)$-dimensional
Lorentz-Minkowski space.
A null hypersurface  in $\R^{n+1}_1$
is a $C^\infty$-immersion whose induced metric
degenerates everywhere.
Such a 
hypersurface is also called a {\it light-like hypersurface}
and is locally foliated by
light-like lines (cf. \cite[Fact 2.6]{AHUY}).
Roughly speaking,
a null hypersurface
is said to be
{\it $L$-complete} if 
each light-like line 
is the entirety of a straight line in $\R^{n+1}_1$
(see Definition \ref{def:965} for details).

In the authors' previous work \cite{AHUY},
it was shown that $L$-complete 
null immersed hypersurfaces in $\R^{n+1}_1$ are totally geodesic.
As is clear from this
previous work, the study of 
global properties of such surfaces, 
considering only immersions is too restrictive.
In this paper, we shall investigate the global behavior of 
null hypersurfaces with singular points
 in $\R^{n+1}_1$. 

More precisely, instead of immersions, we shall introduce null
hypersurfaces as wave fronts (see Section 1).
It can be easily observed that
each hypersurface in the $n$-dimensional 
Euclidean space $\R^n_0$ ($n\ge 2$) induces
an associated parallel family,
and this family can be considered
as a section of a null hypersurface 
in $\R^{n+1}_1$.
For example, the light-cone 
\begin{equation}
\Lambda^n:=\Big\{(t,x_1,\ldots,x_n)\in \R^{n+1}_1\,;\, 
(x_1)^2+\cdots+(x_n)^2=t^2
\Big\}
\end{equation}
is a typical example of an $L$-complete null wave front,
which corresponds to the family of parallel hypersurfaces
of the unit sphere $S^{n-1}$ in $\R^n_0$
centered at the origin.
The origin as the cone-like singular point of $\Lambda^n$
corresponds to the parallel hypersurface of $S^{n-1}$
just shrinking to a point.
In general, 
a null hypersurface may have various singular points.
For example, for each $a\in (0,1]$, we consider an ellipse
\begin{equation}
\gamma(\theta):=(a \cos \theta,\, \sin \theta)\qquad 
(0\le \theta\le 2\pi),
\end{equation}
in the Euclidean plane. Let 
$\mb n(\theta)$ be the leftward unit
normal vector field along $\gamma$.
Then the map $F_a:\R\times (\R/2\pi\Z)\to \R^3_1$
defined by
\begin{equation}\label{eq:Fa}
F_a(t,\theta):=(0,\gamma(\theta))+t(1, \mb n(\theta)) \qquad (a\in (0,1])
\end{equation}
gives a null wave front having cuspidal edges and four swallowtails
whenever $0<a<1$ (the definitions of cuspidal edges and swallowtails
are given in \cite{KRSUY}). 
Each slice of the image of $F_a$ by the
horizontal plane $t=t_0$ corresponds to the parallel curve of 
the ellipse $\gamma$ of equi-distance $t_0$.
When $a=1$, the image of $F_1$ just 
coincides with the light-cone $\Lambda^2$ of $\R^3_1$.

\begin{figure}[htb]
 \begin{center}
\includegraphics[height=3.7cm]{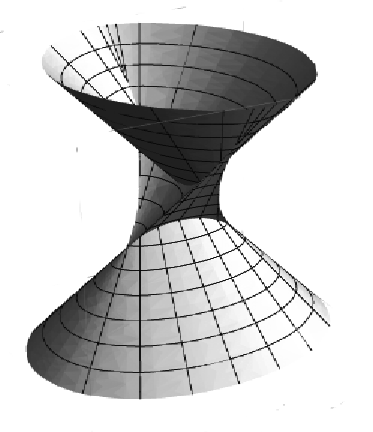}
 \end{center}
\caption{The complete null wave front associated with
the family of parallel curves
of the ellipse of $a=1/2$.}
\label{fig:F0}
\end{figure}

In Section~2, we prove the converse assertion 
(cf. Theorem~\ref{thm:D})
as a fundamental theorem
of null wave fronts, which states that
any $L$-complete null  wave fronts in $\R^{n+1}_1$
are  induced from wave fronts in $\R^n_0$.
As a consequence, we can say that an $L$-complete null wave front
is associated with the parallel family of a wave front in $\R^n_0$.

We give two applications of this fundamental theorem.
In Section 3,
we show a \lq\lq structure theorem of null hypersurfaces''
(cf. Theorem~\ref{thm:II})
which asserts that most of null wave fronts
(for example, real analytic null wave fronts and
$C^\infty$ null wave fronts with a certain genericity)
can be obtained as restrictions of $L$-complete null
wave fronts. 
There are several interesting geometric structures on
null hypersurfaces in $\R^{n+1}_1$ without assuming 
$L$-completeness (cf. \cite{DG, K}). 
The authors hope that
this theorem might play an important role 
in the further study of null hypersurfaces.

On the other hand, 
an $L$-complete null wave front in $\R^{n+1}_1$ 
is called {\it complete}
if its singular set is 
a non-empty compact subset in the domain of definition.	
For example, the light-cone $\Lambda^n$ in $\R^{n+1}_1$ 
and the null wave front
$F_a$ in $\R^3_1$
given in \eqref{eq:Fa} are complete.
In Section 4, as the deepest application of 
the fundamental theorem, we show
that each of complete null wave fronts corresponds to
a parallel family of a closed convex hypersurface in  $\R^n_0$ 
if $n\ge 3$ (cf. Theorem~\ref{ThmE}).  
When $n=2$,
we show that complete null wave fronts 
are induced by parallel families of
locally convex closed regular curves
in the Euclidean plane $\R^2_0$.
If the curve is an ellipse, the null wave front
corresponds to $F_a$ as above (see Figure \ref{fig:F0}). 

It should be remarked that 
the classical four vertex theorem for closed convex curves
implies the existence of
four non-cuspidal edge singular points
on complete null wave fronts with embedded ends in $\R^{3}_1$
(cf. Corollary~\ref{cor:V}).

There is one point to note when reading this paper:
Readers who are interested only in the fundamental theorem for 
$L$-complete null wave fronts
or complete null wave fronts can skip Section 3. 
In fact, Section 4 is devoted to properties of complete null
wave fronts and does not refer to Section 3. 
Appendix A  of this paper is required for Section~2, 
but Appendix B is used for Section 3, 
so such readers also do not need to read Appendix B.

\section{Properties of null wave fronts in
$\R^{n+1}_1$} \label{sec:1} 

We first recall the definition of wave fronts.
Let
$(\R^{n+1})^*$ be the dual vector space of $\R^{n+1}$, 
and we denote by
$
P^*(\R^{n+1})
$
the projective space associated with $(\R^{n+1})^*$.
We let
$$
\pi:(\R^{n+1})^*\setminus \{\mb 0\} \to P^*(\R^{n+1})
$$
be the canonical projection.
Then, for each element
$\zeta\in P^*(\R^{n+1})$,
there exists a linear function
$\omega:\R^{n+1}\to \R$
(which is an element of $(\R^{n+1})^*\setminus \{\mb 0\}$) 
such that $\pi(\omega)=\zeta$.
Since the kernel of
$\omega$
is invariant under non-zero scalar multiplications,
the $n$-dimensional subspace
$$
\op{Ker}(\zeta):=\op{Ker}(\omega)
$$
is well-defined.

In this paper, we set $r=\infty$ or $r=\omega$
and \lq\lq$C^r$'' means smoothness if $r=\infty$ and
real analyticity if $r=\omega$.
We fix a $C^r$-differentiable $n$-manifold $M^n$.

\begin{definition}
Let $F:M^{n}\to \R^{n+1}$ be a $C^r$-map.
Then the map $F$ is called
a 
{\it $C^r$-wave front} (or simply a {\it $C^r$-front})
if for each $p\in M^n$, there exist a neighborhood $U$ of $p$
and a $C^r$-map
$
\tilde \alpha:U\to (\R^{n+1})^*\setminus\{\mb 0\}
$
such that
\begin{enumerate}
\item 
$
\alpha:=\pi\circ \tilde \alpha
$
satisfies
$
(dF)_q(\mb v)\in \op{Ker}(\alpha_q)
\,\, (q\in U,\,\, \mb v \in T_qM^n)
$,
and
\item 
$
L:=(F,\alpha):U\to \R^{n+1}\times P^*(\R^{n+1})
$
is an immersion.
\end{enumerate}
Moreover, if we can take $U=M^n$, the map $F$ is said to be
{\it co-orientable}.
In this case, $\alpha:M^n\to P^*(\R^{n+1})$ is called the
{\it Gauss map} of $F$, 
and the map $\tilde \alpha:M^n\to (\R^{n+1})^*\setminus\{\mb 0\}$ 
is called the {\it lift} of $\alpha$.
\end{definition}

If a $C^r$-wave front
$F$ is not co-orientable, taking the double covering 
$\hat \pi:\hat M^n\to M^n$, the composition $F\circ \hat \pi$ becomes a
co-orientable wave front.
So without loss of generality, we may assume that
$F$ itself is co-orientable.

We let $\R^{n+1}_1$ be Lorentz-Minkowski $(n+1)$-space
of signature $(-+\cdots+)$,
and  denote by
$\inner{\,}{\,}$ 
the canonical Lorentzian inner product on $\R^{n+1}_1$.

\begin{proposition}\label{prop:F-Null}
Let $F:M^{n}\to \R^{n+1}_1$ be a 
$($co-orientable$)$ $C^r$-wave front.
Then there exists a vector field $\hat \xi$
without zeros
$($which can be considered as a map 
$\hat \xi:M^n\to \R^{n+1}_1\setminus\{\mb 0\})$ 
such that
$$
\inner{\hat \xi_p}{dF(\mb v)}=0
\qquad (p\in M^n,\,\, \mb v \in T_pM^n).
$$
\end{proposition}

A vector field $\hat \xi$ along $F$ given in 
Proposition~\ref{prop:F-Null}
is called a {\it normal vector field} along $F$.

\begin{proof}
Let $\tilde \alpha$ be a lift of the Gauss map of $F$.
Since $\inner{\,}{\,}$  is non-degenerate,
there exists a vector field $\hat \xi$ on $M^n$ along $F$
such that
\begin{equation}\label{eq:HatXi}
\inner{\hat \xi_p}{\mb x}
=\tilde \alpha_p(\mb x) \qquad (\mb x\in T_{F(p)}\R^{n+1}_1).
\end{equation}
Since $\alpha_p \ne \mb 0$ for each $p$, 
the vector field $\hat \xi$ has no zeros on $M^n$.
\end{proof}

We then introduce \lq null wave fronts\rq\ as follows:

\begin{definition}\label{def:437}
In the setting of Proposition~\ref{prop:F-Null},
the $C^r$-wave front $F:M^n\to \R^{n+1}_1$ is said to 
be {\it null} or
{\it light-like} if
$\hat\xi_p$ points in the light-like direction in $\R^{n+1}_1$
at each $p\in M^n$.
\end{definition}

We denote by $(\,,\,)_E$
the canonical positive definite
inner product on $\R^{n+1}(=\R^{n+1}_1)$.

\begin{definition}[$E$-normalized normal vector field]
Let $F:M^n\to \R^{n+1}_1$
be a (co-orientable) null $C^r$-wave front.
A normal vector field $\hat \xi$ 
along $F$ is said to be
{\it $E$-normalized} if
$\hat \xi$ points in the future direction
and satisfies
$|\hat \xi|_E=\sqrt{2}$, where
$$
|\mb v|_E:=\sqrt{(\mb v,\mb v)_E} \qquad (\mb v\in \R^{n+1}_1).
$$
Since such a vector field $\hat \xi$ is uniquely determined,
we denote it by $\hat \xi_E$. 
\end{definition}

\begin{remark}
One can replace $\sqrt{2}$ with any positive constant $c$. However,
the choice of $c:=\sqrt{2}$ makes sense in the following reason:
As we will show in Theorem~\ref{prop:make},
a hypersurface $f:\Sigma^{n-1}\to \R^n_0$ in Euclidean space 
$\R^n_0$ with unit normal vector field $\nu$ induces 
null wave fronts $F_\pm:\R\times\Sigma^{n-1}\to \R^{n+1}_1$ 
whose $E$-normalized normal vector fields are given by
$\hat \xi:=(1,\pm \nu)$.
\end{remark}

We can always take the $E$-normalized normal vector field
for a given co-orientable null wave front $F$
as follows:
By Proposition~\ref{prop:F-Null},
we can take a normal vector field $\hat \xi$ along $F$.
Since $F$ is light-like, $\hat \xi$ gives a light-like vector field
along $F$. 
By replacing $\hat \xi$ by $-\hat \xi$, we may assume that
$\hat \xi$ points in the future direction, and
$$
\hat\xi_E:=\frac{\sqrt{2}}{|\hat \xi|_E}\hat \xi
$$
gives the desired vector field.

\begin{lemma}\label{lem:k-lemma0}
Let $F:M^n\to \R^{n+1}_1$
be a $($co-orientable$)$ null wave front.
Then 
\begin{equation}\label{eq:LF}
\mc L_F:=(F,\hat \xi_E):M^n\to \R^{n+1}_1\times \R^{n+1}_1
\end{equation}
is an immersion into $\R^{n+1}_1\times \R^{n+1}_1$.
\end{lemma}

We call this $\mc L_F:M^n\to \R^{n+1}_1\times \R^{n+1}_1$
the {\it canonical lift} of $F$.

\begin{proof}
We denote by  $S^n(\sqrt{2})$ the sphere of radius $\sqrt{2}$ centered at the origin
in $\R^{n+1}_1$ with respect to the metric $(\,,\,)_E$.
Since
$$
S^n(\sqrt{2})\ni \mb v \mapsto \pi(\inner{\mb v}{*})\in P^*(\R^{n+1})
$$
gives the double covering of $P^*(\R^{n+1})$, 
it can be easily observed that $\mc L_F$ is an immersion into 
$\R^{n+1}_1\times S^n(\sqrt{2})$
if and only if $(F,\pi\circ \tilde \alpha):M^n\to \R^{n+1}\times P^*(\R^{n+1})$
is an immersion, where $\tilde\alpha$
is the map $\tilde \alpha:M^n\to (\R^{n+1})^*\setminus\{\mb 0\}$ 
induced by $\hat \xi$ given in \eqref{eq:HatXi}.
\end{proof}

The following assertion gives
a characterization of null wave fronts:

\begin{proposition}\label{prop:536}
Let $F:M^n\to \R^{n+1}_1$ be a $C^r$-map.
Then
$F$ is a $($co-orientable$)$ null wave front
if and only if there exists a vector field
$\hat \xi$ along $F$ defined on $M^n$ such that 
\begin{enumerate}
\item 
$(\hat \xi,\hat \xi)_E=2$,
\item
$\hat \xi$ is pointing in the future light-like direction, and
\item $(F,\hat \xi)$ 
is an immersion into $\R^{n+1}_1\times \R^{n+1}_1$
satisfying $\langle \hat \xi_p,dF_p(\mb v)\rangle=0$
for each $\mb v\in T_pM^n$ $(p\in M^{n})$.
\end{enumerate}
\end{proposition}

\begin{proof}
We have already seen that any null wave front in  $\R^{n+1}_1$
uniquely induces
an $E$-normalized normal vector field $\hat \xi_E$.
So it is sufficient to show the converse.
Suppose that
$\hat \xi$ is a vector field
along $F$ satisfying the above three conditions.
We let
$\tilde \alpha:M^n\to (\R^{n+1})^*\setminus\{\mb 0\}$ 
be the map given in 
\eqref{eq:HatXi}.
By (1) and (3), 
$(F,\pi\circ \tilde \alpha):M^n\to \R^{n+1}\times P^*(\R^{n+1})$
is an immersion by the same reason
as in the proof of
Lemma~\ref{lem:k-lemma0}.
\end{proof}

\begin{lemma}\label{lem:xi-M^n}
Let $F:M^n\to \R^{n+1}_1$ be a null $C^r$-immersion
with a normal vector field $\hat \xi$
on $M^n$ as in Proposition \ref{prop:536}.
Then for each $p\in M^n$,
there exists a unique tangent vector $\xi_p\in T_pM^n$
such that $dF(\xi_p)=\hat \xi_p$.
Moreover, the correspondence $p \mapsto \xi_p$
is $C^r$-differentiable.
\end{lemma}

\begin{proof}
Since $dF_p(T_pM^n)$ is an 
$n$-dimensional vector subspace of $T_{F(p)}\R^{n+1}_1$
at each $p\in M^n$, it holds that
$$
dF_p(T_pM^n)=\Big\{
\hat{\mb v}\in \R^{n+1}_1\,;\, \inner{\hat {\mb v}}{\hat \xi_p}=0\Big\}.
$$
Since $\hat \xi_p\in dF_p(T_pM^n)$ and
$dF_p$ is injective,
the desired vector $\xi_p$ should be
$$
\xi_p:=(dF_p)^{-1}(\hat \xi_p)\qquad (p\in M^n).
$$
Since $(dF_p)^{-1}(\hat \xi_p)$
depends smoothly on $p$,
we obtain the assertion.
\end{proof}

Later, we will show that the assertion of Lemma~\ref{lem:xi-M^n} is extended
for null $C^r$-wave fronts
 (cf. Theorem~\ref{prop:10}).
The next assertion is 
a weak version of Lemma~\ref{lem:xi-M^n}
for null wave fronts.

\begin{proposition}\label{prop:k-lemma-m1}
Let $F:M^n\to \R^{n+1}_1$
be a 
$($co-orientable$)$ null $C^r$-wave front
and 
$\hat \xi$ 
a normal vector field  
along $F$.
Then, for each $p\in M^n$,
there exists a tangent vector $\vect{v}\in T_pM^n$
such that $(dF)_p(\vect{v})=\hat \xi_p$.
\end{proposition}

\begin{proof}
We may assume that $\hat \xi=\hat \xi_E$.
If $F$ gives an immersion at $p$, 
then the assertion
follows from Lemma~\ref{lem:xi-M^n}.
So we may assume that
$p$ is a singular point $($i.e. is a point where 
$F$ is not an immersion$)$ of $F$.
We can take a local coordinate system $(u_1,\ldots,u_n)$
of $M^n$ centered at $p$ such that
the family of vectors
$$
(\partial_{u_1})_p\,\,,\ldots,\,\,(\partial_{u_r})_p \qquad (r>0)
$$
spans the kernel of $(dF)_p$,
where 
$\partial_{u_j}:=\partial/\partial u_j$
($j=1,\ldots,n$).
By setting
$$
F_{u_{i}}:=dF(\partial_{u_i})\qquad (i=1,\ldots,n),
$$
the vectors
$
F_{u_{r+1}}(p)\,\,,\ldots,\,\,F_{u_{n}}(p)
$
are linearly independent.
We denote by
$V$ the subspace of $\R^{n+1}_1$ spanned by these vectors.
We set
\begin{equation}\label{eq:eta}
\hat \eta_i^{}:=d\hat \xi_E(\partial_{u_i}) \qquad (i=1,\ldots, n).
\end{equation}
{If we think $F$ and $\hat \xi_E$ are 
column vector-valued functions, 
we can consider}
the $(2n+2)\times n$-matrix
\begin{equation}\label{eq:M0}
M_0:=\pmt{
F_{u_{1}}            & \cdots & F_{u_{r}}              & F_{u_{r+1}} & \cdots & F_{u_{n}} \\
\hat \eta_1^{} & \cdots   & \hat \eta_r^{} & \hat \eta_{r+1}^{} & \cdots 
& 
\hat \eta_n^{}
}
\end{equation}
at $p$ as the Jacobi matrix of the map
$(F,\hat \xi_E)$ of $M^n$ into $\R^{n+1}_1\times \R^{n+1}_1$,
which can be computed as
$$
M_0=\pmt{
\mb 0            & \cdots & \mb 0              & F_{u_{r+1}}(p) & \cdots & F_{u_{n}}(p) \\
\hat \eta_1^{}(p) & \cdots   & \hat \eta_r^{}(p) & \hat \eta_{r+1}^{}(p) & \cdots 
& \hat \eta_n^{}(p) 
}.
$$
Since $F$ is a wave front,
the matrix $M_0$ is of rank $n$
and so 
$$
\hat \eta_1^{}(p)\,\,,\ldots,\,\, \hat \eta_r^{}(p)
$$
are linearly independent at $p$, which gives  
a basis of the vector space defined by
\begin{equation}\label{eq:sp-W}
W:=\left\{\sum_{i=1}^r a_i \hat \eta_i^{}(p)\,;\, a_1,\ldots,a_r\in \R\right\}.
\end{equation}
Suppose $(\hat \xi_E)_p$ does not belong to $V$.
We set $N:=\R(\hat \xi_E)_p$, that is,
it is the 1-dimensional vector space generated by $(\hat \xi_E)_p$.
By the following Lemma~\ref{lem:VW},
$W\cap N=\{\mb 0\}$ holds, and
$V$, $W$ and $N$ satisfy the assumption of
Lemma~\ref{prop:VWL} in Appendix A.
So we have $W\cap V=\{\mb 0\}$. Since $W$ is of dimension $r$
and $V$ is of dimension $n-r$, 
the direct sum $V+W+N$ coincides with $\R^{n+1}_1$. 
Then $(\hat \xi_E)_p$ vanishes 
because $(\hat \xi_E)_p$ is perpendicular to $V,W$ and $N$,
a contradiction. Thus, we can find a vector $\vect{v}\in T_pM^n$
such that $(\hat \xi_E)_p=dF_p(\vect{v})\in V$.
\end{proof}

\begin{lemma}\label{lem:VW}
The vector space $W$ given in
\eqref{eq:sp-W}
is perpendicular to $dF_p(T_pM^n)$
in $\R^{n+1}_1$.
Moreover, $(\hat \xi_E)_p$ does not belong to $W$.
\end{lemma}

\begin{proof}
As in 
the proof of Proposition~\ref{prop:k-lemma-m1},
we denote by
$V$ the subspace of $\R^{n+1}_1$ spanned by
$F_{u_{r+1}}(p)\,\,,\ldots,\,\,F_{u_{n}}(p)$.
Using the notation
\eqref{eq:eta},
for each $i\in \{1,\ldots,r\}$
and $j\in \{r+1,\ldots,n\}$,
we have
\begin{align*}
\inner{F_{u_j}}{\hat \eta_i^{}} 
&={\inner{F_{u_j}}{\hat \xi_E}}_{u_i}-\inner{F_{u_ju_i}}{\hat \xi_E} 
=-\inner{F_{u_ju_i}}{\hat \xi_E}
\\
&=-\inner{F_{u_i}}{\hat \xi_E}_{u_j}+\inner{F_{u_i}}{d\hat \xi_E(\partial_{u_j})}
=\inner{F_{u_i}}{\hat \eta_j^{}}=0 
\end{align*}
at $p$, 
where we used the fact $F_{u_i}(p)=\mb 0$.
By this computation, we have that
\begin{equation}\label{eq:vw}
\inner{v}{w}=0 \qquad (v\in V,\,\, w\in W),
\end{equation}
proving the first assertion.

We  prove the second assertion.
If $(\hat \xi_E)_p\in W$, then we can write
$
\hat \xi_E=\sum_{i=1}^r a_i(\hat \xi_E)_{u_i}
$
at $p$.
Then it holds that
$$
2=\Big((\hat \xi_E)_p,(\hat \xi_E)_p\Big)_E
=\sum_{i=1}^r 
a_i\Big(\big((\hat \xi_E)_{u_i}\big)_p,\big(\hat \xi_E\big)_p\Big )_E=0,
$$
a contradiction.
\end{proof}

We next prepare the following:

\begin{proposition} \label{lem:Fd}
Let $F:M^n\to \R^{n+1}_1$ be a null $C^r$-wave front
and $\hat \xi_E$ its associated $E$-normalized normal vector field. 
Suppose that $p\in M^n$ is a singular point of $F$.
Then 
\begin{enumerate}
\item 
for each $\delta\in \R$,
\begin{equation}\label{eq:FDelata}
F_\delta:=F+\delta\, \hat \xi_E 
\end{equation}
is a null wave front defined on $M^n$, and
\item 
there exists a positive number $\epsilon_0$ such that 
$F_\delta$ $(0<|\delta|<\epsilon_0)$ is an immersion at $p$.
\end{enumerate}
\end{proposition}

\begin{remark}
Later, we will show that
the image of $F_\delta$ coincides with that of
$F$ under a suitable assumption of $F$
(see Corollary \ref{cor:condC}).
\end{remark}

\begin{proof}
Since $\hat \xi_E$ is a light-like vector field, it is
a normal vector field of $F_\delta$.
We set
$$
M_\delta:=\pmt{
(F_\delta)_{u_{1}} & \cdots & (F_\delta)_{u_{r}} 
& (F_\delta)_{u_{r+1}} & \cdots & (F_\delta)_{u_{n}} \\
(\hat \xi_E)_{u_{1}} & \cdots & (\hat \xi_E)_{u_{r}} & 
(\hat \xi_E)_{u_{r+1}} & \cdots & (\hat \xi_E)_{u_{n}}} 
$$
for $\delta\in \R$.
Since $M_\delta$ ($\delta\ne 0$)
is obtained 
by adding the second row to the first row of $M_0$,
the pair $(F_\delta,\hat \xi_E)$ gives an immersion
of $M^n$ into $\R^{n+1}_1\times \R^{n+1}_1$ 
if and only if $(F,\hat \xi_E)$ is an immersion. So (1) is proved. 

We next  prove (2).
By Proposition~\ref{prop:k-lemma-m1},
there exists a vector $\vect{v}\in T_pM^n$
such that $dF_p(\vect{v})=(\hat \xi_E)_p$.
We then
take a local coordinate system $(u_1,\ldots,u_n)$
of $M^n$ centered at $p$ such that
$$
(\partial_{u_1})_p\,\,,\ldots,\,\,(\partial_{u_r})_p \qquad 
(r>0)
$$
belong to the kernel of $(dF)_p$
and $(\partial_{u_n})_p=\vect{v}$.
We let $V$ be the subspace of $\R^{n+1}_1$, which  is  spanned by
$F_{u_{r+1}}(p)\,\,,\ldots,\,\,F_{u_{n-1}}(p)$.
We consider
the $r$-dimensional
vector space $W$ given by
\eqref{eq:sp-W}.
By
Lemma~\ref{lem:VW},
it can be easily checked that
the three vector subspaces $V,\,\, W$ and $N:=\R\hat\xi_E$
satisfy the assumption of 
Lemma~\ref{prop:VWL} in Appendix~A. 
So we have 
$$
V\cap W=V\cap N=W\cap N=\{\mb 0\}.
$$
In particular, the vectors
$$
 (\hat \xi_E)_{u_1},\,\,\ldots, (\hat \xi_E)_{u_r},\,\,
F_{u_{r+1}}\,\,,\ldots,\,\,F_{u_{n}}(=\hat \xi_E)
$$
are linearly independent at $p\in M^n$.
We fix a non-zero real number $\delta$.
Then we have
\begin{align*}
m(\delta):=
&\op{rank}\Big(dF_\delta(\partial_{u_1}),\dots, dF_\delta(\partial_{u_r}),
                    dF_\delta(\partial_{u_{r+1}}),
\dots, dF_\delta(\partial_{u_n})\Big)\\
&
=
\op{rank}\Big(\delta d\hat\xi_E(\partial_{u_1}),
\dots, \delta d\hat\xi_E(\partial_{u_r}),
                    dF(\partial_{u_{r+1}})+
\delta d\hat\xi_E(\partial_{u_{r+1}}),\\
&\phantom{aaaaaaaaaaaaaaaaaaaaaaaaaaaaaaaaa}\dots, 
dF(\partial_{u_n})+\delta d\hat\xi_E(\partial_{u_{n}})\Big) \\
&
=\op{rank}\Big(d\hat\xi_E(\partial_{u_1}),\dots, 
d\hat\xi_E(\partial_{u_r}),
dF(\partial_{u_{r+1}})+\delta d\hat\xi_E(\partial_{u_{r+1}}),\\
&\phantom{aaaaaaaaaaaaaaaaaaaaaaaaaaaaaaaaa}\dots, 
dF(\partial_{u_n})+\delta d\hat\xi_E(\partial_{u_{n}})\Big). 
\end{align*}
If $|\delta|$ is
sufficiently small,
the right-hand side of the last equality
is equal to $n$ at $p$. So  we can conclude that
$F_\delta$ is an immersion on a sufficiently small neighborhood of $p$
for sufficiently small $|\delta|(>0)$.
\end{proof}

\begin{definition}[$E$-normalized null vector field]
Let $F:M^n\to \R^{n+1}_1$ be a null $C^r$-wave front
and $\hat \xi_E$ its associated $E$-normalized normal vector field. 
If there exists a $C^r$-differentiable
vector field $\xi_E$ defined on $M^n$ 
such that
$
dF(\xi_E)=\hat \xi_E
$,
then $\xi_E$ is called the {\it $E$-normalized null vector field} with
respect to  $F$
(in fact, $\xi_E$ is uniquely determined as follows).
\end{definition}

The following is 
the deepest result in this section,
which asserts the existence of
$E$-normalized null vector field:

\begin{theorem}\label{prop:10}
Let $F:M^n\to \R^{n+1}_1$ be a null $C^r$-wave front
and $\hat \xi_E$ its associated $E$-normalized normal vector field. 
Then 
there exists a unique 
$C^r$-differentiable
vector field $\xi_E$ defined on $M^n$ 
such that
\begin{enumerate}
\item
$
dF(\xi_E)=\hat \xi_E
$
$($that is, $\xi_E$ is the $E$-normalized null vector field$)$, and
\item the image of each integral curve of $\xi_E$
is a part of a light-like line in $\R^{n+1}_1$.
\end{enumerate}
Moreover, 
$dF(T_pM^n)$ is a light-like vector space in $\R^{n+1}_1$.
\end{theorem}

\begin{proof}
Roughly speaking, the key of the proof is that 
the image of the null wave front $F$ is foliated 
by light-like lines
as mentioned in the introduction, and
we will construct a vector field $\xi_E$ defined on $M^n$ 
so that each integral curve of $\xi_E$  corresponds to
the leaf of the foliation as follows.

We fix a point $p\in M^n$ arbitrarily.
By Proposition~\ref{lem:Fd},
for a sufficiently small positive number $\delta$,
there exists
a neighborhood $U$ of $p$
such that $F_\delta$ is an immersion 
on $U$.
Hence, by Lemma~\ref{lem:xi-M^n}, 
there exists a unique
$C^r$-differentiable null vector field $\xi_E$ 
satisfying
$
dF_\delta(\xi_E)=\hat \xi_E
$
on $U$.
We let $\gamma_q(t)$ be
the integral curve of $\xi_E$  
such that $\gamma_q(0)=q$ for $q\in U$.
Since $\hat \xi_E$ is a
null vector field along $F_\delta$,
\cite[Fact 2.6]{AHUY}
implies that
$F_\delta\circ \gamma_q$ parametrizes a segment of
a light-like line.
Thus $F_\delta\circ \gamma_q(t)$ is a 
geodesic in $\R^{n+1}_1$, and so we have
$
D_{\xi_E}dF_\delta(\xi_E)=\mb 0,
$
where $D$ is the Levi-Civita connection of $\R^{n+1}_1$.
So,
$$
dF(\xi_E)=dF_\delta(\xi_E)-\delta (d \hat\xi_E)(\xi_E)
=\hat \xi_E-\delta D_{\xi_E}dF_\delta(\xi_E)=\hat \xi_E
$$
holds at $p$. 
Since the light-like vector $(\hat \xi_E)_q$ ($q\in U$)
lies in $dF_q(T_qM^n)$, the vector space
$(dF)_q(T_qM^n)$ is
light-like. 
Since $p$ is arbitrary, the uniqueness of 
$\xi_E$ implies it can be defined on $M^n$, proving the assertion.
\end{proof}

\section{A fundamental theorem for $L$-complete null wave fronts}

We first define the following \lq\lq $L$-completeness''
of null wave fronts in
$\R^{n+1}_1$, which is analogous to the case of
null immersions given in \cite[Definition 2.7]{AHUY}.

\begin{definition}\label{def:965}
Let $F:M^n\to \R^{n+1}_1$ be a 
(co-orientable)
null $C^r$-wave front.
The map $F$ is called {\it $L$-complete} if
for each $p\in M^n$,
there exists an integral curve $\gamma:\R\to M^n$
of the $E$-normalized null vector field $\xi_E$
passing through $p$ such that
$F\circ \gamma(\R)$ coincides with an entire
light-like line in $\R^{n+1}_1$.
\end{definition}

The following assertion holds:
 
\begin{proposition}
Let $F:M^n\to \R^{n+1}_1$ be a $($co-orientable$)$
null $C^r$-wave front.
Then $F$ is $L$-complete if and only if
its $E$-normalized null vector field 
$\xi_E$ is complete on $M^n$
$($that is, each integral curve of $\xi_E$ is defined on $\R)$.
\end{proposition}

\begin{proof}
Let $\gamma(t)$ be any integral curve of $\xi_E$.
Since $|dF(\xi_E)|_E=\sqrt{2}$,
$\gamma(t)$ is parametrized by an affine parametrization of a light-like line.
So $\gamma$ is defined on $\R$ if and only if
the image of $F\circ \gamma$ coincides with the entirety
of a line.
\end{proof}

We consider the height function
\begin{equation}\label{eq:tau}
\hat\tau:\R^{n+1}_1 \ni \mb v \mapsto -\inner{\mb v}{\mb e_{0}}\in \R
\end{equation}
with respect to the time axis, where
$
\mb e_{0}:=(1,0,\ldots,0)
$.
The level set $\hat\tau^{-1}(0)$ is a space-like hyperplane
which is isometric to the Euclidean $n$-space $\R^n_0$.
So, we frequently use the identification 
\begin{equation}\label{eq:Id}
\R^n_0\ni x:=(x_1,\ldots,x_n) \,\,\longleftrightarrow \,\,
\tilde x:=(0,x_1,\ldots,x_n) \in \hat\tau^{-1}(0).
\end{equation}
Let $\Sigma^{n-1}$ be an $(n-1)$-manifold. 
A  $C^r$-map $f:\Sigma^{n-1}\to \R^{n}_0$
is called a (co-orientable) {\it wave front} if
there exists a unit normal vector field
$\nu$ along $f$ defined on $\Sigma^{n-1}$
such that the map
$\Sigma^{n-1}\ni p \mapsto (f(p),\nu(p))\in \R^{n}_0\times S^{n-1}$
is an immersion, where $S^{n-1}$ is the unit sphere centered
at the origin in $\R^{n}_0$.
We now show that the following representation formula 
of $L$-complete null wave fronts in $\R^{n+1}_1$ from a given 
wave front in $\R^n_0$ as follows:

\begin{theorem}\label{prop:make}
Let $f:\Sigma^{n-1}\to \R^n_0$ be a 
$($co-orientable$)$  $C^r$-wave front  with
unit normal vector field $\nu$.
Then for each choice of
$\sigma\in \{+,-\}$,
the map
$\mc F^f_{\sigma}:\R\times \Sigma^{n-1} \to \R^{n+1}_1$
defined by
\begin{equation}
\mc F^f_{\sigma}(t,x)
:=\tilde f(x)+t \hat \xi_{\sigma}(x), \qquad
\hat \xi_{\sigma}(x):=(1,\sigma\nu(x))
\qquad (t\in \R,\,\, x\in \Sigma^{n-1})
\end{equation}
gives an $L$-complete null wave front in $\R^{n+1}_1$,
where
$
\tilde f(x):=(0,f(x)).
$
Moreover, the regular set of $\mc F^f_{\sigma}$
is dense in $\R\times \Sigma^{n-1}$. 
\end{theorem}

When $f$ is an immersion, this formula is known (see Kossowski
\cite{K}). So this theorem can be considered as its
generalization for wave fronts.
The slice of the image of $\mc F^f_{\pm}$ by
a hyperplane $\{t=c\}$ ($c\in \R$) is congruent to
the image of a parallel hypersurface $f^c$ of $f$.

\begin{remark}
Since
$$
\mc F^f_-(t,x)=\pmt{-1 & 0 & \cdots& 0 \\
          0 & 1 & \cdots& 0 \\
         \vdots& \vdots  & \ddots& \vdots \\
         0 &  0 & \cdots & 1}\, \mc F^f_+(-t,x),
$$
the image of $\mc F^f_-$ is congruent to that of $\mc F^f_+$ in $\R^{n+1}_1$.
So we call $\mc F^f_+$ the {\it normal form} of the
null wave front associated with $f$. 
\end{remark}

\begin{proof}[Proof of Theorem~\ref{prop:make}]
Without loss of generality, we may assume that $\sigma=+$,
and we set
$$
F:=\mc F^f_+,\qquad \hat \xi:=
\hat \xi_{+}
\Big(=\big(1,\nu\big)\Big).
$$
Since  $\hat \xi(x)$ ($x\in \Sigma^{n-1}$)
is orthogonal
to  $dF(T_p(\R\times \Sigma^{n-1}))$ with respect to
the Lorentzian inner product $\inner{\,}{\,}$,
the vector field
$\partial_t:=\partial/\partial t$
is a null vector field of $F$ defined on $\R\times \Sigma^{n-1}$.
We write
\begin{align*}
& F=(F^0,\ldots,F^{n}), 
\qquad \hat \xi=(\hat\xi^0,\ldots,\hat\xi^{n}), \\
&f=(f^1,\ldots,f^{n}), \,\,
\text{ and }\,\,\,
\nu=(\nu^1,\ldots,\nu^{n}).
\end{align*}
In the following discussions,
we think that $F$ and $\nu$ take values in
column vectors.
Then we have
$$
F^{0}=t, \quad \hat\xi^{0}=1,\quad F^i=f^i+t \nu^i,\quad
\hat\xi^i=\nu^i \qquad (i=1,\ldots,n).
$$
We fix $x\in \Sigma^{n-1}$ arbitrarily,
and take a local coordinate system
$(u_1,\ldots,u_{n-1})$ of $\Sigma^{n-1}$ centered at $x$.
Then
$(t,u_1,\ldots,u_{n-1})$ gives a 
local coordinate system of $\R\times \Sigma^{n-1}$.
To show that
$\mathcal{L}_F:=(F,\hat \xi)$
is an immersion,
it is sufficient to show that
$(2n+2)\times n$ matrix
$$
M_t:=\pmt{
F_{t} &
F_{u_1}&
\cdots & F_{u_{n-1}} \\
\hat \xi_{t} &
\hat \xi_{u_1} &
\ldots & \hat \xi_{u_{n-1}} }
$$
is of rank $n$.
Since
\begin{align}\label{eq:Ft1151}
&(F_t(t,x),\,\, F_{u_1}(t,x),\,\, \ldots,\,\, F_{u_{n-1}}(t,x)) \\
\nonumber
&\phantom{aaaaaaaaaaaa}=
\pmt{
1    &   0                       &  \ldots &  0  \\
\nu(x)   & f_{u_1}(x)+t\nu_{u_1}(x)   & \ldots & f_{u_{n-1}}(x)+t \nu_{u_{n-1}}(x)
}
\end{align}
holds, we have
\begin{equation}\label{eq:Mt}
M_t
=
\pmt{
1    &   0                       &  \ldots &  0  \\
\nu   & f_{u_1}+t \nu_{u_1}   & \ldots & f_{u_{n-1}}+t \nu_{u_{n-1}} \\
0    &   0                       &  \ldots &  0  \\
\mb 0 & \nu_{u_1}  & \ldots & \nu_{u_{n-1}}
}.
\end{equation}
Since $f$  is a wave front, we have
$$
n-1=
\op{rank}\pmt{f_{u_1}   & \ldots & f_{u_{n-1}} \\
\nu_{u_1}  & \ldots & \nu_{u_{n-1}}
},
$$
which implies that $M_t$ is of rank $n$, that is, 
$F$ is a null wave front by
Proposition~\ref{prop:536}.
Since we have already shown that
$\partial_t$ points in the null direction,
$F$ is a null wave front.
By definition, it is obvious that $F$ is $L$-complete.

We next show that
$F$ is an immersion on an
open dense subset of $\R\times \Sigma^{n-1}$.
By
\eqref{eq:Ft1151},
the matrix $(F_t,F_{u_1},\ldots,F_{u_{n-1}})$
is of rank $n$ at $(t,x)\in \R\times \Sigma^{n-1}$
if
$$
n-1=
\op{rank} \Big( f_{u_1}(x)+t \nu_{u_1}(x),\,\,   \ldots ,\,\,
f_{u_{n-1}}(x)+t \nu_{u_{n-1}}(x)
\Big).
$$
To prove that this holds for almost all $(t,x)$, 
one can use the fact that 
the parallel hypersurface $f_t:=f+t\nu$ (for fixed $t$)
has a singular point at $x$
if and only if $t$ coincides with
one of the inverse of principal curvatures
of $f$ at $x$:
If a given point $(t,x)\in \R\times \Sigma^{n-1}$ is a singular point of $F$,
then $x$ is a regular point of $f_{t_n}$ for $t_n:=t+1/n$. 
In particular, $(t,x)$ is an accumulation point of the regular 
points $\{(t_n,x)\}_{n=1}^\infty$ of $F$, and we can conclude that
the regular set of $F$ is dense in $\R\times \Sigma^{n-1}$.
\end{proof}

\begin{lemma}\label{prop:Lf-E}
In the setting of
Theorem~\ref{prop:make},
if 
$$
l_f:\Sigma^{n-1}\ni x \mapsto (f(x),\nu(x))\in \R^n_0\times S^{n-1}
$$
is an embedding, then by setting
$
\hat \xi_+(x):=(1,\nu(x)),
$
the map
$$
\mc L_F:\R\times \Sigma^{n-1}\ni (t,x)\mapsto \big(\mathcal F^f_+(t,x),
\hat \xi_+(x)\big)\in 
\R^{n+1}_1\times 
\R^{n+1}_1
$$
is also an embedding, where $F:=\mathcal F^f_+$.
\end{lemma}

\begin{proof}
We remark that $\mc L_F$ ($F:=\mc F^f_+$)
is a map into $\R^{n+1}_1\times (\R^{n+1}_1\setminus\{\mb 0\})$.
Since $F$ is a null wave front, $\mc L_F$ is an immersion
(cf. Proposition \ref{prop:536}).
So it is sufficient to prove that $\mc L_F$ is a homeomorphism
from $M^n$ to $\mc L_F(M^n)$.
For the sake of simplicity, we set $\hat \xi:=\hat \xi_+$. 

We let $(P,\hat \xi)$ be a point in $\mc L_F(\R\times \Sigma^{n-1})$.
Then there exists a unique point
$\Phi(P,\hat \xi)$
at which the straight line $\{P+t \hat \xi\,;\, t\in \R\}$
meets the hyperplane $\hat\tau^{-1}(0)$ in $\R^{n+1}_1$.
We let 
\begin{equation}\label{eq:p0R}
\pi_0:\R^{n+1}_1\to \hat\tau^{-1}(0)
\end{equation}
be the canonical orthogonal projection.
Since $\Phi(P,\hat \xi)$
depends on $P$ and $\hat \xi$
continuously,
the map
$
(P,\hat\xi)\mapsto (\Phi(P,\hat \xi),\pi_0(\hat \xi))
$
is a continuous map
whose image coincides with $l_f(\Sigma^{n-1})$.
So 
$$
\mc L_F(\R\times \Sigma^{n-1})\ni 
(P,\hat \xi)\mapsto (\hat\tau(P),l_f^{-1}\circ \Phi(P,\hat \xi))\in \R\times \Sigma^{n-1}
$$
is well-defined, and gives the 
inverse map of $\mc L_F$.
So $\mc L_F$ is an embedding. 
\end{proof}

Conversely, we can prove the following:

\begin{theorem} \label{thm:D}
Let $F:M^n\to \R^{n+1}_1$ be a $($co-orientable$)$
$L$-complete null $C^r$-wave
front,
then  there exists a co-orientable $C^r$-wave front
$f:\Sigma^{n-1}\to \R^{n}_0$
and a diffeomorphism
$\Phi:M^n\to \R\times \Sigma^{n-1}$
such that $F\circ \Phi^{-1}$ coincides with
the null wave front $F^f_+$
as in Theorem \ref{prop:make}.
\end{theorem}

\begin{proof}
Let  $\xi_E$ be the $E$-normalized null vector field of $F$.
By  Theorem \ref{prop:10}, 
the image of each integral curve $\gamma(t)$ of $\xi_E$ 
is a light-like geodesic.
So the identity
\begin{equation}\label{eq:Lambda}
D_{\xi_E} \hat \xi_E=D_{\gamma'(t)} dF(\gamma'(t))=
\mb 0 
\end{equation}
holds on $M^n$, where $\hat \xi_E:=dF(\xi_E)$ and $D$ is the
Levi-Civita connection of $\R^{n+1}_1$.

Since
$|\hat \xi_E|_E=\sqrt{2}$ and $\hat \xi_E$ points
in the light-like direction,
we have
$
d\hat \tau(dF(\xi_E))\ne \mb 0
$,
in particular, by applying the implicit function theorem,
the zero-level set 
$$
\Sigma^{n-1}:=(\hat\tau\circ F)^{-1}(0)\,(\subset M^n)
$$
is an embedded hypersurface of $M^n$.
Let $\{\phi_t\}_{t\in \R}$ be the
one parameter group of transformations 
on $M^n$ associated with $\xi_E$.
As in \cite[Proposition A.3]{AHUY},
the map defined by
\begin{equation}\label{eq:Psi}
\Psi:\R\times \Sigma^{n-1}\ni (t,p)\mapsto \phi_t(p)\in M^n
\end{equation}
is an immersion. 
Since (cf. \eqref{eq:Lambda})
$
\hat \gamma(t):=F(\phi_t(p))\,\, (t\in \R)
$
is a light-like geodesic in $\R^{n+1}_1$,
we obtain the expression
\begin{equation}\label{eq:F}
F\circ \Psi(t,p)=F\circ \phi_t(p)
=F(p)+t \hat \xi_E(p) \qquad (t\in \R,\,  p\in \Sigma^{n-1}).
\end{equation}
In particular,
$
\R\ni t \mapsto F\circ \Psi(t,p)\in \R^{n+1}_1
$
is an injection, and so,
$$
\R\ni t \mapsto \phi_t(p)\in M^n
$$
is also an injection.
We suppose that
$\Psi$ is not injective
and 
$
\Psi(t,p)=\Psi(s,q)
$
holds.
Since $\phi_{s-t}(q)=p$, 
the point $q$ lies on the integral curve 
passing through $p$.
Since
$\hat\gamma(=F\circ \gamma)$ is a straight line,
$\gamma$ meets $\Sigma^{n-1}$ exactly once.
So we can conclude $p=q$ and $s=t$.
Thus $\Psi$ is an injection.
On the other hand,
since $F$ is $L$-complete,
any integral curve of $\xi_E$
must meet $\Sigma^{n-1}$. 
So $\Psi$ is bijective
and then it becomes
a diffeomorphism.
If we set
$\hat \xi_E(p)=(a(p),\nu(p))$ ($p\in M^n$),
then we have
\begin{equation}\label{eq:aa1}
2=(\hat \xi_E(p),\hat \xi_E(p))_E=a(p)^2+|\nu(p)|^2.
\end{equation}
Since
\begin{equation}\label{eq:aa2}
0=\langle\hat \xi_E(p),\hat \xi_E(p)\rangle=-a(p)^2+|\nu(p)|^2,
\end{equation}
we have
$
a(p)^2=|\nu(p)|^2=1,
$
which implies that $\nu$ is a unit normal vector field of 
$\tilde f:=F|_{\Sigma^{n-1}}$
in the space-like hyperplane $\hat\tau^{-1}(0)$.
We fix $\mb v\in T_p\Sigma^{n-1}$ arbitrarily, then
we can write
$
(d\tilde f)_p(\mb v)=(0,\mb a)
$
where  $\mb a\in \R^{n}_0$.
Since $\hat \xi_E=(1,\nu)$, we have
\begin{equation}\label{eq:aa3}
0=\inner{(d\tilde f)_p}{\hat \xi}=\mb a\cdot \nu,
\end{equation}
where the dot denotes the canonical Euclidean inner product
on $\R^n_0$, where $\hat\tau^{-1}(0)=\{0\}\times \R^n_0$.
So $\nu$ is a
unit normal vector field of $f$.

Finally, we show that $\tilde f$ is a wave front defined on $\Sigma^{n-1}$:
Since $F$ is a null wave front, the pair  $\mc L_F:=(F,\hat \xi_E)$ 
is an immersion on $M^n$ into $\R^{2n+2}$.
Then its restriction 
$$
(\tilde f,\hat \xi_E|_{\Sigma^{n-1}}):\Sigma^{n-1}\to \R^{n+1}_1\times \R^{n+1}_1
$$ 
is also an immersion.
Moreover, since $
\xi_E|_{\Sigma^{n-1}}=(1,\nu),
$
the map
$$
(f,\nu):\Sigma^{n-1}\to \R^{n}_0\times \R^{n}_0
$$
is an immersion (where $\tilde f=(0,f)$), and so
$f:\Sigma^{n-1}\to \R^{n}_0$ is a co-orientable wave front.
Summarizing the above discussions, 
\eqref{eq:F} can be written as
\begin{align*}\label{eq:F2}
F\circ \Psi(t,x)&=F\circ \phi_t(x)
=F(x)+t \hat \xi_E(x) \\ 
&=\tilde f(x)+t(1,\nu(x))=F^f_+(t,x) \qquad ((t,x)\in \R\times \Sigma^{n-1}).
\end{align*}
By setting $\Phi:=\Psi^{-1}$, we obtain the assertion.
\end{proof}

\begin{corollary}\label{cor:condB}
The regular set of an
$L$-complete
null $C^r$-wave  front 
$F:M^n\to \R^{m+1}_1$
is dense in
$M^n$.
\end{corollary}

\begin{proof}
We have just shown that such a null wave front can be reparametrized as
a normal form.
So, the assertion follows from the last statement of
Theorem~\ref{prop:make}.
\end{proof}

\begin{corollary}\label{cor:condC}
Let $F:M^n\to \R^{n+1}_1$ be an $L$-complete 
null $C^r$-wave
 front. Then, 
for each $\delta\in \R$
the parallel hypersurface $F_\delta$
given in  \eqref{eq:FDelata}
is also an $L$-complete null wave front 
and has the same image as $F$.
\end{corollary}

\begin{proof}
By Theorem \ref{thm:D},
there exists a co-orientable $C^r$-wave front
$f:\Sigma^{n-1}\to \R^{n}_0$
and a diffeomorphism
$\Phi:M^n\to \R\times \Sigma^{n-1}$
such that $F\circ \Phi^{-1}=F^f_+$ holds on $M^n$.
We let $\hat \xi_E$ be the
$E$-normalized normal vector field
along $F$. Then $\hat \xi_E\circ \Phi^{-1}$ is
the
$E$-normalized normal vector field of $F^f_+$.
So we have that
\begin{align*}
F_\delta\circ \Phi^{-1}(t,x)
&=F\circ \Phi^{-1}(t,x)+\delta \hat \xi_E\circ \Phi^{-1}(t,x) \\
&=\Big(\tilde f(x)+t \hat \xi_E\circ \Phi^{-1}(t,x)\Big)
+\delta \hat \xi_E\circ \Phi^{-1}(t,x) \\
&=\tilde f(x)+(t+\delta)\hat \xi_E\circ \Phi^{-1}(t,x)
=F\circ \Phi^{-1}(t+\delta,x),
\end{align*}
which proves the assertion.
\end{proof}

\section{A structure theorem of null wave fronts}

In this section, we give a structure theorem of null wave fronts
without assuming $L$-completeness.
Roughly speaking, a null
 wave front is foliated by line segments, 
so by extending each line segment to a whole line and 
connecting them appropriately, we will obtain an $L$-complete null wave front.

\begin{figure}[htb]
 \begin{center}
\includegraphics[height=3.7cm]{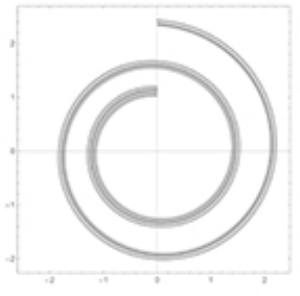} \qquad
\includegraphics[height=3.7cm]{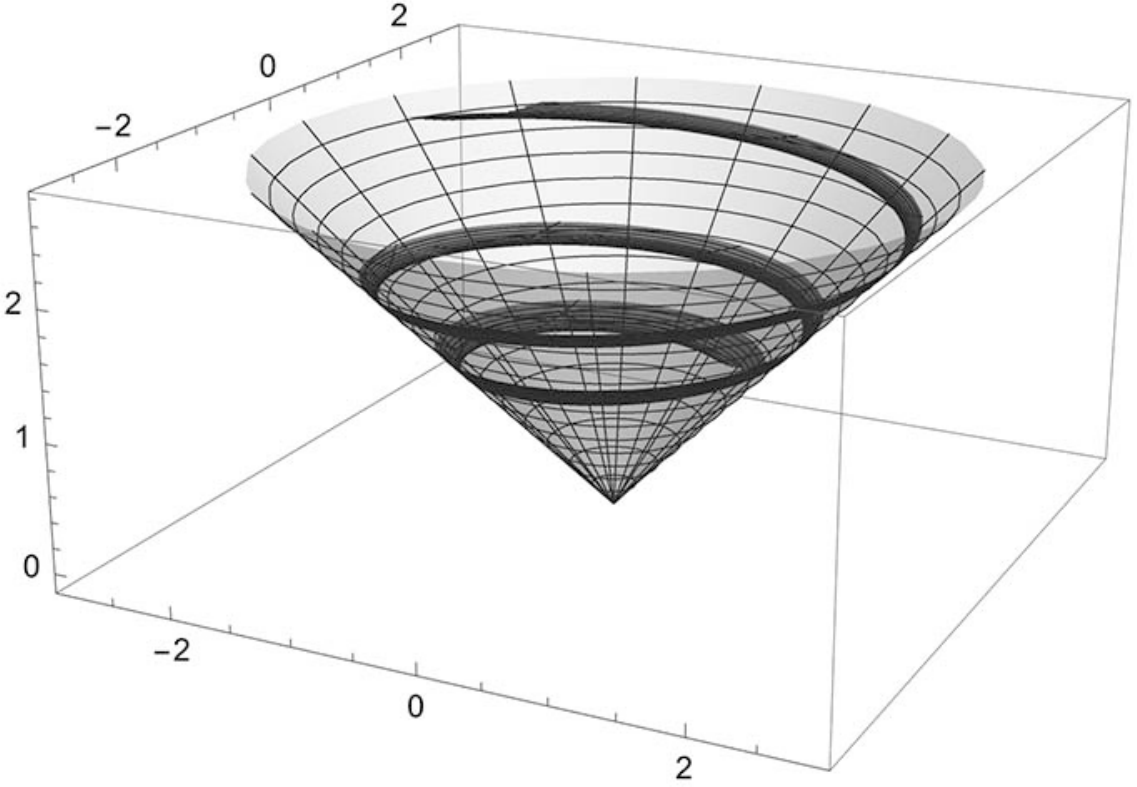}
 \end{center}
\caption{
The tubular neighborhood $\mc U$ of the logarithmic 
spiral $\gamma$ and
its image $F(\mc U)$ in the light-cone $\Lambda^2$.}
\label{fig:1630}
\end{figure}

Here is one example to illustrate the structure theorem: 
We consider the $1/2$-tubular neighborhood $\mc U$ 
of the logarithmic spiral 
(see Figure \ref{fig:1630}, left)
$$
\gamma(t)=e^{t/15}(\cos t,\sin t) \qquad (\pi/2<t<4\pi+\pi/2)
$$
in the $xy$-plane, 
and consider the image $F(\mc U)$ with respect to 
the graph $F(x,y):=(x,y,\sqrt{x^2+y^2})$.
Then $F(\mc U)$ is an open subset of the light-cone $\Lambda^2$
(see Figure \ref{fig:1630}, right). 
Since $F(\mc U)$ is a ruled surface, 
we can extend each of the ruling light-like lines
 to both sides, 
and obtain an $L$-complete null wave front $\check F$ 
as an $L$-completion of $F(\mc U)$.
However, this surface $\check F$ is different from the light-cone $\Lambda^2$.
In fact, each light-like line as a generator of the
ruled surface $\check F$ meets $F(\mc U)$ several times
and $\check F$ can  be regarded 
as a kind of double covering the light-cone $\Lambda^2$.
To produce the actual $\Lambda^2$ from the extension of $F(\mc U)$,
we need to consider the quotient space with an appropriate equivalence relation 
in the image of $\check F$.
In this example, the resulting $L$-completion is
the light-cone $\Lambda^2$, which is, of course, a Hausdorff space.
However, in the general situation, 
in order for the domain of definition
of the resulting $L$-complete null wave front 
to be a Hausdorff space, it is necessary to assume 
appropriate conditions on the original null wave front.

From now on, we fix a (co-orientable) null $C^r$-wave  front $F:M^n\to \R^{n+1}_1$
which may not be $L$-complete.
Let
\begin{equation}\label{eq:tau3}
\tau:M^n\ni p \mapsto \hat\tau\circ F(p)\in \R,
\end{equation}
be the restriction of the
height function $\hat\tau$ to the hypersurface $F$.
he following lemma will play an important role:

\begin{lemma} \label{prop:F-R}
Let $F:M^n\to \R^{n+1}_1$ be a 
$($co-orientable$)$ null $C^r$-wave front, and let
$p$ be a point in $M^n$. Then, for each open neighborhood $\tilde U$ of $p$, 
there exist 
\begin{itemize}
\item a neighborhood $U(\subset \tilde U)$ of $p$,
\item a positive number $\epsilon_U^{}\in (0,\infty]$ 
and a real number $t_U^{}$,
\item a $C^r$-differentiable $(n-1)$-submanifold $\Sigma_U$ of $U$,
\item a surjective $C^r$-submersion $\rho_U^{}:U\to \Sigma_U$, and
\item 
a $C^r$-wave front $\tilde g_U^{}:\Sigma_U\to \{0\}\times \R^n_0$
\end{itemize}
such that
\begin{enumerate}
\item 
We set
$
\hat \xi_U(x):=\hat \xi_E|_{\Sigma_U}(x)
$
$(x\in \Sigma_U)$ and denote by
$
\hat \xi_U=(1,\nu_U^{})
$.
Then $\nu_U^{}(x)$ gives a unit normal vector field of $\tilde g_U^{}$ in $\R^n_0$
$($by regarding $\tilde g_U^{}$ is a map into $\R^n_0)$.
\item 
We set
\begin{align} \label{eq:GU}
&G_U:\R\times \Sigma_U\ni (t,x) \mapsto \tilde g_U^{}(x)+t \hat \xi_U(x)\in \R^{n+1}_1
\qquad (x\in \Sigma_U,\,\, t\in \R), \\ \nonumber
&\tau_U(q):=\tau(q)\qquad (q\in U).
\end{align}
Then the map given by
$$
\Phi_U(q):=(\tau_U^{}(q),\rho_U^{}(q))\qquad (q\in U)
$$
is a diffeomorphism between $U$ 
and $I_U\times \Sigma_U$ 
satisfying
\begin{equation}\label{eq:FqU}
F(q)=G_U\circ \Phi_U(q)
=\tilde g_U^{}\circ \rho_U(q)+\tau_U^{}(q) \hat \xi_U\circ\rho_U^{}(q) \qquad (q\in U),
\end{equation}
where $I_U:=(-\epsilon_U^{}+t_U,\epsilon_U^{}+t_U)$ if $\epsilon_U^{}\ne \infty$
and $I_U=\R$ if  $\epsilon_U=\infty$.
\item
The map given by
$$
\mc L_U:\R\times \Sigma_U\ni (t,x)\mapsto (G_U(t,x),\hat \xi_U(x))\in \R^{n+1}_1\times \R^{n+1}_1, 
$$
is an embedding
satisfying
$$
\mc L_F(q)=\mc L_U\circ \Phi_U(q) \qquad (q\in U).
$$
\item
The map defined by
$$
\tilde l_U:\Sigma_U\ni x \mapsto (\tilde g_U^{}(x),\hat \xi_U(x))\in (\{0\}\times \R^{n}_0)\times \R^{n+1}_1
$$
is an embedding.
\end{enumerate}
In this setting, if $F$ itself is $L$-complete and $\tilde U:=M^n$, then 
the above assertions hold by setting $I_U:=\R$.
\end{lemma}

\begin{definition}
In this setting,
if
$\mc L_U:\R\times \Sigma_U\to 
\R^{n+1}_1\times \R^{n+1}_1$ 
is an embedding,
$U$ is called 
an {\it $F$-adapted neighborhood} of
$p$.
 Moreover, 
$(U,\Sigma_U,\tau_U^{},\rho_U^{},\tilde g_U^{},\hat \xi_U,I_U)$ 
is called
the {\it fundamental $F$-data} at $p$. 
\end{definition}

\begin{proof}[Proof of Lemma~\ref{prop:F-R}]
Let  $\xi_E$ be the $E$-normalized null vector field of $F$.
Since
$\hat \xi_E(\ne \mb 0)$ points
in the light-like future direction,
we have
$
d\tau(\xi_E)> 0
$.
In particular, by the implicit function theorem,
the level set 
$$
\tilde \Sigma^{n-1}:=
\tau^{-1}(\tau(p))\cap \tilde U\,(\subset M^n)
$$
is an embedded hypersurface of $\tilde U$.
We set 
\begin{equation}\label{eq:1332}
t_0:=\tau(p).
\end{equation}
Since $\xi_E$ is transversal to
the hypersurface $\tilde\Sigma^{n-1}$,
there exist
\begin{itemize}
\item an open interval of the form $J:=(-\epsilon,\epsilon)$ ($\epsilon>0$)
or $J:=\R$,
\item a connected open submanifold $\Sigma^{n-1}$ of $\tilde \Sigma^{n-1}$
satisfying $p\in \Sigma^{n-1}$, 
and
\item an injective $C^r$-immersion
$$
\tilde \Psi:J\times \Sigma^{n-1}
\ni (t,x) \mapsto \phi_t(x)\in \tilde U
$$
\end{itemize}
such that (cf. \cite[Proposition~A.3]{AHUY})
the map $t \mapsto \phi_t(x)$ ($x\in \Sigma^{n-1}$)
gives an integral curve of $\xi_E$
satisfying $\phi_{0}(x)=x$.
In this setting,
if $F$ is $L$-complete and $\tilde U=M^n$, then $\xi_E$ is a complete
vector field of $M^n$ and we can set $J:=\R$.

Since $\tilde \Psi$ is an injective immersion
between manifolds of the same dimension,
by the inverse function theorem, $\tilde \Psi$ is
an open map and so 
\begin{equation}\label{eq:1377}
U:=\tilde \Psi(J\times \Sigma^{n-1})(\subset \tilde U)
\end{equation}
is a neighborhood of $p$, and $\tilde \Psi$
gives  a diffeomorphism from $J\times \Sigma^{n-1}$ to 
$U$. 
Moreover, if $F$ is $L$-complete and $\tilde U:=M^n$, 
the map $\tilde \Psi$
gives a diffeomorphism between $\R\times \Sigma^{n-1}$
and $M^{n}$ by the same reason why $\Psi$ 
is a diffeomorphism in the proof of
Theorem~\ref{thm:D}.

We define a map by
$$
\rho:U\ni q\mapsto \pi_2\circ \tilde \Psi^{-1}(q)\in \Sigma^{n-1},
$$
where 
$\pi_2$ is the canonical projection of
$J \times \Sigma^{n-1}$ onto $\Sigma^{n-1}$.
Since  $\tau\circ \phi_{t}(x)=t+t_0$ for $x\in \Sigma^{n-1}$
(cf. \eqref{eq:1332}), 
$\rho$ is a surjective submersion satisfying
\begin{equation}\label{eq:t0}
\big(\tau(q)-t_0,\rho(q)\big)=\tilde \Psi^{-1}(q) \qquad (q\in U).
\end{equation}

By Theorem~\ref{prop:10}, 
$F\circ \phi_t(x)$ ($t\in J$,\, $x\in \Sigma^{n-1}$)
lies on a light-like straight line,
and so, the vector $\hat \xi_E\circ \phi_t(x)\in \R^{n+1}_1$
does not depend on  the parameter $t$.
So we can write
$$
(\hat \xi_E(x):=)\hat \xi_E\circ \phi_t(x)=(a(x),\nu(x))\qquad (x\in \Sigma^{n-1}),
$$
where $a(x)>0$ (since $\hat \xi_E$ points in the future direction).
We regard
$\hat\tau^{-1}(t_0)$ is a Euclidean $n$-space.
By the same argument as in the proof of Theorem~\ref{thm:D}
(cf. \eqref{eq:aa1}, \eqref{eq:aa2} and \eqref{eq:aa3}),
we have $a=\pm 1$ and can write
\begin{equation}\label{eq:xn}
\hat \xi_E(x)=(1,\nu(x)), \quad |\nu(x)|=1
\qquad (x\in \Sigma^{n-1})
\end{equation}
such that $(0,\nu)$ gives  a unit normal vector field of
the map
$$
\tilde f:=F|_{\Sigma^{n-1}}:\Sigma^{n-1}
\to \hat\tau^{-1}(t_0)(\subset \R^{n+1}_1).
$$
So $f$ is a frontal in $\hat\tau^{-1}(t_0)$.
Moreover, since $F(\phi_t(x))$ ($x\in \Sigma^{n-1}$)
parametrizes a light-like straight line,
we have that
\begin{equation}\label{eq:F-Psi}
F\circ \tilde \Psi(t,x)=\tilde f(x)+t \hat \xi_E(x) \qquad (t\in J,\,\, x\in \Sigma^{n-1}).
\end{equation}

We next show that $\tilde f$ is a wave front in 
the hyperplane $\hat\tau^{-1}(t_0)$.
Since $F$ is a null wave front 
in $\R^{n+1}_1$, the pair  $\mc L_F:=(F,\hat \xi_E)$ 
is an immersion from 
$M^n$ into $\R^{n+1}_1\times \R^{n+1}_1$
(cf. Lemma \ref{lem:k-lemma0}).
By
\eqref{eq:F-Psi},
we can write
$$
\mc L_F\circ \tilde \Psi(t,x)=\big(\tilde f(x)+ t \hat \xi_E(x),\hat \xi_E(x)\big) 
\qquad ((t,x)\in J\times \Sigma^{n-1}).
$$
By setting $\tilde f(x)=(t_0, f(x))$ ($x\in \Sigma^{n-1}$),
the rank of the matrix 
\eqref{eq:Mt} for $t=t_0$ 
satisfies 
$$
\op{rank}\pmt{
1    &   0                       &  \ldots &  0  \\
\nu   & f_{u_1}+t \nu_{u_1}   & \ldots & f_{u_{n-1}}+t \nu_{u_{n-1}} \\
0    &   0                       &  \ldots &  0  \\
\mb 0 & \nu_{u_1}  & \ldots & \nu_{u_{n-1}}
}
=
\op{rank}\pmt{1 &       0 & \ldots & 0 \\
     \nu &  f_{u_1} & \ldots & f_{u_{n-1}} \\
     0 &       0 & \ldots & 0 \\
     \bf 0 &  \nu_{u_1} & \ldots & \nu_{u_{n-1}} 
}.
$$
Since the right-hand side is equal to $n$, the matrix
$$
\pmt{
     f_{u_1} & \ldots & f_{u_{n-1}} \\
     \nu_{u_1} & \ldots & \nu_{u_{n-1}}
}
$$
is of rank $n-1$ at each point of $\Sigma^{n-1}$.
So $\tilde f$ is a wave front on $\Sigma^{n-1}$.

We consider the map $\tilde g:\Sigma^{n-1}\to \hat\tau^{-1}(0)$ given by
$$
\tilde g(x):=\tilde f(x)-t_0 \hat \xi_E(x)\qquad (x\in \Sigma^{n-1}),
$$
which corresponds to the parallel hypersurface of $f$
having signed equi-distance $-t_0$ in $\R^n_0$. 
Then $\tilde g$ is a wave front in $\hat\tau^{-1}(0)$ whose unit normal 
vector field $\nu$ satisfies
$$
(1,\nu(x))=\hat \xi_E|_{\Sigma^{n-1}}(x)\qquad (x\in \Sigma^{n-1}).
$$
Then, by Theorem~\ref{prop:make},
\begin{equation}\label{eq:G}
G(t,x):=\tilde g(x)+ t 
(1,\nu(x))\qquad (x\in \Sigma^{n-1},\,\, t\in \R)
\end{equation}
is an  $L$-complete null wave front.
We set
$$
I_U:=
\begin{cases}
(-\epsilon+t_0,\epsilon+t_0) & \text{if $\epsilon \ne \infty$}, \\
\R & \text{if $\epsilon=\infty$},
\end{cases}
$$
and consider the following diffeomorphism
$$
\Upsilon: I_U\times \Sigma^{n-1} \ni (t,x)\mapsto 
(t-t_0,x)\in J\times \Sigma^{n-1}.
$$
By \eqref{eq:t0}, it holds that
\begin{equation}\label{eq:1530}
q=\tilde \Psi(\tau(q)-t_0,\rho(q))=\tilde \Psi\circ \Upsilon(\tau(q),\rho(q))\qquad (q\in U).
\end{equation}
By this with \eqref{eq:F-Psi} and \eqref{eq:G},
we have 
\begin{align*}
F(q)&=F\circ \tilde \Psi\circ \Upsilon(\tau(q),\rho(q))=F\circ \tilde \Psi(\tau(q)-t_0,\rho(q)) \\
&=\tilde f\circ \rho(q)+(\tau(q)-t_0)\hat\xi_E\circ \rho(q) \\
&=\tilde g\circ\rho(q)+\tau(q)\hat\xi_E\circ\rho(q) =G(\tau(q),\rho(q))
\end{align*}
for each $q\in U$.
Thus,
by setting 
\begin{align*}
&\Phi_U:=(\tilde \Psi\circ \Upsilon)^{-1},\quad  
\epsilon_U^{}:=\epsilon, \quad
t_U:=t_0,
\quad \rho_U^{}:=\rho, \\
&\nu_U:=\nu,\quad
\hat \xi_U:=(1,\nu_U^{}), \quad
\Sigma_U:=\Sigma^{n-1},\quad G_U:=G,\quad
\tilde g_U^{}:=\tilde g,
\end{align*}
we obtain the desired fundamental data:
In fact, if we set
\begin{align*}
& \mc L_U:\R\times \Sigma^{n-1}\ni
(t,x) \mapsto \big(G_U(t,x),\hat \xi_U(x)\big) 
\in \R^{n+1}_1\times \R^{n+1}_1,\\
& \tilde l_U:\Sigma^{n-1}\ni
x \mapsto \big(\tilde g_U^{}(x),\hat \xi_U(x)\big) 
\in (\{0\}\times \R^n_0)\times \R^{n+1}_1, 
\end{align*}
then
the two maps $\tilde l_U$ and $\mc L_U$ are immersions,
since we have already shown that
$\tilde g_U^{}$ and $G_U$ are wave fronts.

Since every immersion is locally an embedding,
if we choose the open
neighborhood $U(\subset \tilde U)$ of $p$
to be sufficiently small 
(that is, we choose a sufficiently small $\Sigma^{n-1}$, see \eqref{eq:1377}), 
then $\tilde l_U$ gives an embedding.
Then, by
Lemma~\ref{prop:Lf-E},
$\mc L_U$ is also an embedding. 
\end{proof}

By Lemma~\ref{prop:F-R},
for each point $p\in M^n$, there exist
fundamental $F$-data
$$
(U_p,\Sigma_p,\tau_p,\rho_p,\tilde g_p,\hat \xi_p,I_p)
$$
such that $U_p$ is an $F$-adapted neighborhood of $p$
giving an $L$-complete null wave front (cf.~\eqref{eq:GU})
$$
G_p:\R\times \Sigma_p
\ni(t,x) \mapsto \tilde g_p(x)+t \hat \xi_p(x)\in \R^{n+1}_1
$$
and an embedding 
\begin{equation}\label{eq:1781}
\tilde l_p:\Sigma_p\ni x \mapsto 
(\tilde g_p(x),\hat\xi_p(x))\in (\{0\}\times \R^n_0)\times
 \R^{n+1}_1 \subset \R^{n+1}_1\times \R^{n+1}_1,
\end{equation}
where $\hat \xi_p(x)=(1,\nu_p(x))$.
Since $M^n$ satisfies the second axiom of the countability,
there exist an at most countable set $\mc A$
and a family of fundamental $F$-data
$$
\{(U_\lambda,\Sigma_\lambda,\tau_\lambda,\rho_\lambda,
\tilde g_\lambda^{},\hat \xi_\lambda,I_\lambda)\}_{\lambda\in \mc A}
$$
such that
\begin{enumerate}
\item for each $\lambda\in \mc A$, there exists $p_\lambda\in U_\lambda$ 
satisfying 
\begin{align*}
& (U_\lambda,
\Sigma_\lambda,\tau_\lambda,\rho_\lambda,\tilde g_\lambda^{},\hat \xi_\lambda,I_\lambda)
:=(U_p,\Sigma_{p_\lambda},\tau_{p_\lambda},\rho_{p_\lambda},\tilde g_{p_\lambda}^{},
\hat \xi_{p_\lambda},
I_{p_\lambda}), \\
& \tilde g_\lambda^{}:=\tilde g_{p_\lambda}^{},
\end{align*}
\item the map $\tilde l_\lambda:\Sigma_\lambda\to \R^{n+1}_1\times \R^{n+1}_1$
defined by
$$
\tilde l_\lambda(x):=(\tilde g_\lambda^{}(x),\hat \xi_\lambda(x))
\qquad (\lambda\in \mc A)
$$
is an embedding (cf. (4) of Lemma \ref{prop:F-R}),
\item $M^n=\bigcup_{\lambda\in \mc A}U_\lambda$.
\end{enumerate}
In particular, $\{(\Sigma_{\lambda},\tilde l_\lambda)\}_{\lambda\in \mc A}$ 
is a family of embeddings defined on $(n-1)$-dimensional manifolds.
To ensure that the domain of definition of the resulting $\check F$ 
is a Hausdorff space,  we give the following definition:

\begin{definition}\label{def:1688}
In the above setting,
 $F$ is said to be {\it admissible}
if  $\{(\Sigma_{\lambda},\tilde l_\lambda)\}_{\lambda\in \mc A}$ 
is an admissible family as in Definition \ref{def:1668}
in Appendix~B.
\end{definition}

We prove the following:

\begin{theorem}[A structure theorem of null wave fronts]\label{thm:II}
Let $F:M^n\to \R^{n+1}_1$ be an admissible 
null $C^r$-wave front
$($if $F$ is real analytic, i.e. $r=\omega$, it is admissible, see
Lemma~\ref{lem:1653}$)$. Then there exist
\begin{itemize}
\item  a $C^r$-differentiable $(n-1)$-manifold $\check \Sigma^{n-1}$,
\item  a $C^r$-immersion $\mc I:M^n\to \R\times \check \Sigma^{n-1}$,
\item an $L$-complete null wave front
$
\check F:\R\times \check \Sigma^{n-1} \to \R^{n+1}_1
$
written in the normal form and its canonical lift
$
\check {\mc L}:\R\times \check \Sigma^{n-1} \to \R^{n+1}_1\times \R^{n+1}_1
$
\end{itemize}
such that 
$$
\check F\circ \mc I(q)=F(q),\qquad
\check {\mc L}\circ \mc I(q)=\mc L_F(q)\qquad (q\in M^n).
$$
If $\mc L_F$ is an embedding $($see Remark \ref{rmk:1942}$)$, then
$\mc I$ gives a diffeomorphism between $M^n$ and
$\mc I(M^n)$.
Moreover, if $F$ is $L$-complete, then
$\mc I$ is a surjection.
\end{theorem}

\begin{remark}\label{rmk:1942}
Since $F$ is a null wave front,
its canonical lift $\mc L_F$ is 
an immersion.
So, the assumption that $\mc L_F$ is an embedding 
is not so restrictive. 
On the other hand,
even if
$\mc L_F$ is an embedding, the admissibility of $F$
does not hold in general, see Example \ref{pC}.
\end{remark}

\begin{proof}
We consider the disjoint union
$
\mathcal S:=\coprod_{\lambda\in \mc A} \Sigma_\lambda,
$
and give a relation $x\sim y$ for 
$x,y\in \mathcal S$
so that $x\sim y$ implies that
there exists a pair $(\lambda,\mu)\in\mc A\times \mc A$ of indices
such that
\begin{itemize}
\item $\Sigma_\lambda$ is a neighborhood  
of $x$, 
\item $\Sigma_\mu$ is a neighborhood of $y$, and
\item $x$ is $(\tilde l_\lambda,\tilde l_\mu)$-related to $y$
in the sense of Definition \ref{def:UV1601}.
\end{itemize} 
As seen in the proof of
Proposition \ref{prop:1671}, 
the symbol $\sim$ gives an equivalence relation,
and $\check\Sigma^{n-1}:=\mathcal S/{\sim}$ is an $(n-1)$-manifold.
Then 
$
\pi:\mathcal S
\to \check \Sigma^{n-1}
$
is
the canonical projection
as an open map.
We set 
$$
\phi_\lambda:=\pi|_{\Sigma_\lambda}.
$$
Then $\{(\Sigma_\lambda,\phi_\lambda)\}_{\lambda\in \mc A}$
is the differentiable structure of $\check\Sigma^{n-1}$ as
shown in the proof of Proposition \ref{prop:1671}. 
Moreover, 
an immersion
$
\check l:\check\Sigma^{n-1} \to \R^{n+1}_1\times \R^{n+1}_1
$
is induced which satisfies 
$$
\check l \circ \phi_\lambda^{}(x)=(\tilde g_\lambda^{},\hat \xi_\lambda)
(=\tilde l_\lambda)
\qquad (x\in \Sigma_\lambda,\,\,
\lambda\in \mc A).
$$
We can write
$$
\check l(\check x)=(\tilde g(\check x),\hat\xi(\check x))
\qquad (\check x\in \check\Sigma^{n-1}).
$$
Then, by definition, 
we have 
$$
\tilde g\circ \phi_\lambda^{}(x)=\tilde g_\lambda^{}(x),\quad
\hat \xi\circ \phi_\lambda(x)=\hat \xi_\lambda(x) \qquad (x\in \Sigma_\lambda)
$$
So, if we set 
\begin{align*}
\check F(t,\check x)&:=\tilde g(\check x)+ t \hat \xi (\check x)
\qquad 
((t,\check x)\in \R\times \check\Sigma^{n-1}), \\
G_\lambda(t,x)&:=\tilde g_\lambda^{}(x)+t\hat \xi_\lambda(x)
\qquad (t\in \R,\, x\in \Sigma_\lambda),
\end{align*}
then it holds that
\begin{align}\label{eq:1968}
\check F(t,\phi_\lambda(x))
&=\tilde g\circ \phi_\lambda(x)+ t \hat \xi \circ \phi_\lambda^{}(x) \\ \nonumber
&=\tilde g_\lambda^{}(x)+ t \hat \xi_\lambda(x)
=G_\lambda(t,x)
\end{align}
for $t\in \R$ and $x\in \Sigma_\lambda$.
So, if we consider
the  maps  defined by
\begin{align*}
&\check{\mc L}:\R\times \check\Sigma^{n-1}\ni (t,\check x)\mapsto  
\Big (\tilde g(\check x)+ t \hat \xi (\check x),
\hat \xi (\check x)\Big)\in \R^{n+1}_1\times \R^{n+1}_1, \\
&\mc L_\lambda:\R\times \Sigma_\lambda\ni
(t,x) \mapsto \big(G_\lambda(t,x),\hat \xi_\lambda(x)\big) 
\in \R^{n+1}_1\times \R^{n+1}_1,
\end{align*}
then we have
$$
\check{\mc L}(t,\phi_\lambda(x))
=\mc L_\lambda(t,x)\qquad (t\in \R,\,\,x\in \Sigma_\lambda).
$$
We set
$$
\check \Sigma_\lambda:=\phi_\lambda(\Sigma_\lambda).
$$
Since $\mc L_\lambda$
is an embedding (see (3) of Lemma  \ref{prop:F-R}), 
then the restriction $\check{\mc L}|_{\R\times \check \Sigma_\lambda}$
is also an embedding for each $\lambda\in \mathcal A$.
In particular,
$\check{F}$ is an $L$-complete null wave front on 
$\R\times \check \Sigma^{n-1}$.

We then fix $p\in M^n$ arbitrarily, and assume that
$p$ belongs to $U_\lambda$ for some $\lambda\in \mc A$.
By \eqref{eq:FqU} and \eqref{eq:1968},
we have
\begin{align*}
F(q)
&=\tilde g_\lambda^{}\circ \rho_\lambda^{}(q)+
\tau_\lambda^{}(q)\, \hat \xi_\lambda\circ \rho_\lambda(q) \\
&=G_\lambda(\rho_\lambda^{}(q),\tau_\lambda^{}(q))
=\check F(\tau_\lambda^{}(q),\phi_\lambda^{}\circ \rho_\lambda^{}(q))
\qquad (q\in U_\lambda).
\end{align*}
So it holds that
\begin{equation}\label{eq:2004}
\mc L_F(q)=\check {\mc L}(\tau_\lambda^{}(q),\phi_\lambda^{}\circ \rho_\lambda^{}(q)) 
\qquad (q\in U_\lambda).
\end{equation}
Since $\mc L_F$ and $\check{\mc L}$ are immersions,
the map
\begin{equation} \label{eq:2016}
\mc I_\lambda:U_\lambda\ni q\mapsto (\tau_\lambda^{}(q),\phi_\lambda^{}\circ \rho_\lambda^{}(q))
\in \R\times \check \Sigma_\lambda
\end{equation} 
is also an immersion.

If $p\in U_\mu$
for some other index $\mu\in \mc A$, then 
the immersion 
$ 
\mc I_\mu:U_\mu\to \R\times \check \Sigma_\mu
$
is also induced.
Since $\check {\mc L}|_{\R\times\check \Sigma_\lambda}$ 
and $\check {\mc L}|_{\R\times \check \Sigma_\mu}$
are embeddings,
so is $\check {\mc L}|_{\R\times (\check \Sigma_\lambda\cap\check \Sigma_\mu)}$.
Thus $\mc I_\lambda$ coincides with $\mc I_\mu$ on 
$U_\lambda\cap U_\mu$, 
and the map
$ 
\mc I:M^n\to \R\times \check \Sigma^{n-1}
$
is canonically induced. By \eqref{eq:2004} and
\eqref{eq:2016},
$$
\mc L_F(q)=\check {\mc L}\circ \mc I(q) \qquad (q\in M^n)
$$
holds. In particular,
$\check F\circ \mc I=F$ also holds on $M^n$.
We now consider the case that $\mc L_F$ is an embedding.
Suppose that $\mc I(p)=\mc I(q)$ ($p,q\in M^n$).
Then we have 
$$
\mc L_F(p)=\check {\mc L}\circ \mc I(p)=\check {\mc L}\circ \mc I(q)=\mc L_F(q).
$$
Since $\mc L_F$ is an embedding, 
we have $p=q$, proving the injectivity of $\mc I$.
Since $\mc I$ is an immersion between same dimensional manifolds, it is
a diffeomorphism.

If $F$ is $L$-complete, then 
$\tau_\lambda^{}(U_\lambda)=\R$
and 
$\rho_\lambda^{}(U_\lambda)=\Sigma_\lambda$
for each $\lambda\in \mc A$, which imply that
the map $\mc I_\lambda$
is surjective (cf. \eqref{eq:2016}).
So we can conclude that $\mc I$ is a surjection.
\end{proof}

\begin{example}
We set $\R^2_*:=\R^2\setminus \{(0,0)\}$,
and consider a null immersion
$$
F:\R^2_*\ni (u,v)\mapsto (u^2+v^2, 2uv,u^2-v^2)\in (\R^3_1;t,x,y),
$$
whose image lies on the light-cone $\Lambda^2$ passing through the origin
in $\R^3_1$. 
In this case, the image of the $L$-completion 
of $F$ is the light-cone $\Lambda^2$ and so the map $F$
covers twice on $\Lambda^2\cap \{t>0\}$ in $\R^3_1$.
In particular, the induced map
$\mc I$ is not a diffeomorphism.
This example shows that
if we drop the condition that $\mc L_F$ is injective,
the injectivity of $\Phi$ does not follow, in general.
\end{example}

\begin{figure}[htb]
 \begin{center}
\includegraphics[height=2.7cm]{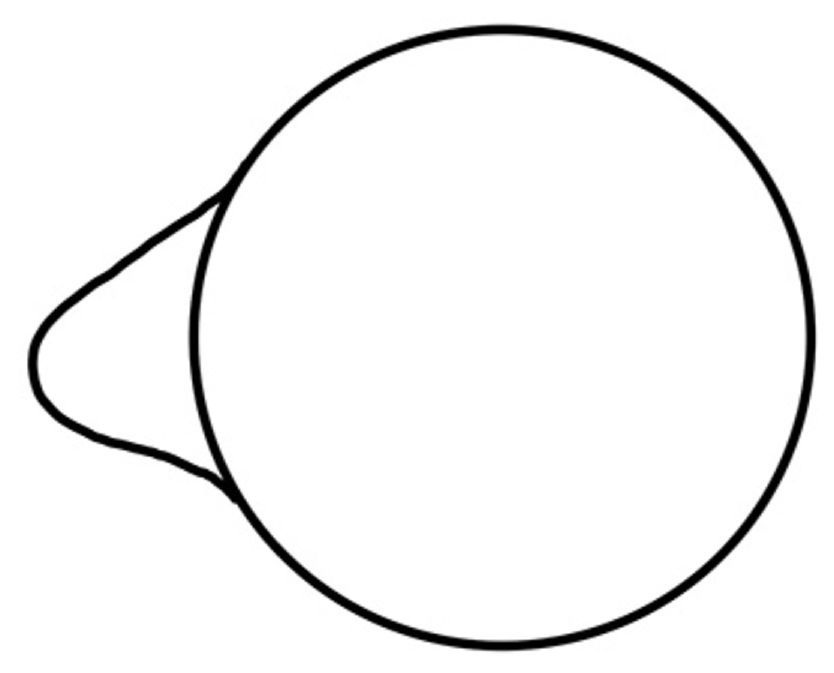}

\caption{The image of the curve $\gamma$
in Example \ref{pC}.}
\label{fig:gamma}
\end{center}
\end{figure}

\begin{example}\label{pC}
Consider a plane curve (cf. Figure \ref{fig:gamma})
$$
\gamma(t):=e^{\omega(t)}(\cos t,\sin t) \qquad (0\le t\le 4\pi),
$$
where $\omega(t)$ is a $C^\infty$-function into $[0,1]$
such
that $\omega(t)=0$ for $t\not \in (3\pi-\epsilon,3\pi+\epsilon)$
and $\omega(t)>0$ for $t\in (3\pi-\epsilon,3\pi+\epsilon)$,
where $\epsilon$ is a sufficiently small positive number.
Then $\gamma$ generates the $L$-complete null wave front $\mathcal F^\gamma_+$
(cf. Theorem~\ref{prop:make}).
By the construction of $\gamma$,
$\mathcal F^\gamma_+$ is not admissible, since
the image of $\mathcal F^\gamma_+$ has non-transversal intersections.
We then set
$$
\Gamma(t):=(0,\gamma(t))+ t(1,\nu(t))\qquad (0\le t\le 4\pi),
$$
where $\nu(t)$ is a unit normal vector field of $\gamma(t)$.
The image $C:=\Gamma([0,4\pi])$ of $\Gamma$ 
is an embedded spiral-shaped curve
on image of $\mathcal F^\gamma_+$.
Consider a sufficiently small tubular neighborhood $\mc U_C$
of $C$ in the image of $\mathcal F^\gamma_+$.
Then $\mc U_C$ is an embedded null surface,
because $\Gamma$ has no self-intersections.
If we think that $F$ is the inclusion map $\mc U_C\hookrightarrow \R^3_1$,
then the $L$-completion of $F$ is $\mathcal F^\gamma_+$.
So, this example shows that the admissibility 
of the $L$-completion of $F$ in 
Theorem~\ref{thm:II}
is independent of
the embeddedness of the canonical lift $\mc L_F$.

Moreover, 
if we apply our $L$-completion procedure 
for $\mc U_C$, then the resulting quotient space 
is not a Hausdorff space, which
implies the necessity of the admissibility condition as in
Definition \ref{def:1688}.
\end{example}

\section{Classification of complete null wave fronts}

In this section, we define  \lq\lq completeness'' 
of null wave fronts and prove a structure theorem of
them. 
The content of this section is completely independent of  Section~3.

\begin{definition}\label{def:complete}
A null $C^r$-wave front $F:M^n\to \R^{n+1}_1$ 
is said to be {\it complete}
if 
\begin{itemize}
\item $F$ is $L$-complete, and
\item the singular set of $F$ is 
      a non-empty compact subset of $M^n$.
\end{itemize}
\end{definition}

\begin{remark}\label{Rmk:2014}
As proved in \cite{AHUY},
if $F$ is $L$-complete and its singular set
is empty, then the image $F(M^n)$
is a subset of a light-like hyperplane of $\R^{n+1}_1$,
which is a trivial case.
So we only consider the case that
the singular set of $F$ is 
non-empty, in the above definition.
\end{remark}

\begin{remark}
When $n=2$, a null wave front $F$ in $\R^3_1$ is
a surface with vanishing Gaussian curvature 
(cf. \cite[Appendix]{AUY}).
If $F$ is complete, then it is complete 
as a flat front 
in the Euclidean 3-space 
in the sense of  Murata-Umehara \cite{MU}.
\end{remark}

We prove the following:

\begin{theorem}\label{ThmE}
Let $F$ be a null $C^r$-wave  front
in $\R^{n+1}_1$.
If $F$ is complete, 
then it can be 
reparametrized as a normal form $\mc F^f_+$,
where 
$f:\Sigma^{n-1}\to \R^{n}_0$ 
is an immersion defined on
a compact $(n-1)$-manifold $\Sigma^{n-1}$
whose principal curvatures are non-zero and all same sign
everywhere in $\R^n_0$.
In particular, if $n\ge 3$, then $f$
is a compact convex hypersurface in $\R^{n}_0$.
On the other hand, if $n=2$,
then $f$ is a closed locally convex regular curve in $\R^2_0$.
\end{theorem}

\begin{proof}
We let $F:M^n\to \R^{n+1}_1$ be a complete null wave front.
By  Theorem~\ref{thm:D},
we may assume that
there exist
\begin{itemize}
\item an $(n-1)$-manifold $\Sigma^{n-1}$,
\item a diffeomorphism $\Phi:M^n\to \R\times \Sigma^{n-1}$, 
and\item  a co-orientable wave front $f:\Sigma^{n-1}\to \R^{n}_0$
\end{itemize}
such that $F\circ \Phi^{-1}$ coincides with the normal form 
$\mc F^f_+$ given in Theorem~\ref{prop:make}.
We  consider the height function
$\hat \tau:\R^{n+1}_1 \to \R$
given in \eqref{eq:tau}.
Since the singular set of $F$ is compact,
for sufficiently large $t_0$, the restriction 
$$
\tilde f:=F\circ \Phi^{-1}|_{\Phi(\tau^{-1}(t_0))}
:\Sigma^{n-1}\to \hat \tau^{-1}(t_0)
$$
is an immersion such that $\tilde f(x)=(t_0,f(x))$.
Without loss of generality, we may assume that
$t_0=0$ and $F$ satisfies
\begin{equation}\label{eq:Ftp}
F\circ \Phi^{-1}(t,x)=\tilde f(x)+t \hat \xi_E(x), 
\quad
\hat \xi_E(x):=(1,\nu(x))
\qquad (t\in \R,\,\, x\in \Sigma^{n-1}),
\end{equation}
where $\nu$ can be considered as
the unit normal vector field along the immersion $f$.
So, we may assume that $F$ is defined on $\R\times \Sigma^{n-1}$.
Since $f$ is co-orientable and is an immersion,
$\Sigma^{n-1}$ is orientable.
We let  
$$
\lambda_1\le \cdots \le \lambda_{n-1}
$$
be principal curvature functions of  $f$.
Here, $\{\lambda_i(x)\}_{i=1}^{n-1}$
are eigenvalues of 
a symmetric matrix associated with the shape operator of $f$
at $x$.
Since the characteristic polynomial of a real symmetric matrix
consists only of real roots, 
the well-known fact that the roots of a polynomial 
depend continuously on its coefficients 
implies that 
$\{\lambda_i\}_{i=1}^{n-1}$
can be considered as 
a family of real-valued continuous
functions defined on $\Sigma^{n-1}$.
We first show that each $\lambda_i$ never changes sign:
Suppose that
there exists a sequence of points $\{x_k\}_{k=1}^\infty$
which converges to a point $x_\infty$ such that
$$
\lambda_i(x_k)\ne 0 \,\, \text{ and }\,\, \lambda_i(x_\infty)=0.
$$
Then
$$
\left(\frac1{\lambda_i(x_k)},x_k\right)\in \R\times \Sigma^{n-1}
\qquad (k=1,2,3,\ldots)
$$
are singular points of $F$ which are unbounded on $\R\times \Sigma^{n-1}$,
contradicting the compactness of the singular set of $F$.
Thus, each $\lambda_i$ 
as a continuous function on $\Sigma^{n-1}$
has no zeros unless it is identically zero.
By Remark \ref{Rmk:2014}, $f$ is not a part of a hyperplane in $\R^n_0$.
So, there exists an integer $r$ ($1\le r\le n-1$)
such that 
$
\lambda_{i_s}\,\, (s=1,\ldots,r)
$
are not identically zero, where $\{i_1,\ldots,i_r\}$
is a subset of $\{1,\ldots,n-1\}$.
By Hartman's product theorem
(cf. \cite[Page 347]{kn2}),
$\Sigma^{n-1}$ is a product
of a compact manifold and $\R^l$ ($l:=n-1-r$).
For each $s\in \{1,\ldots,r\}$,
we can define a 
continuous map $\psi_s:\Sigma^{n-1}\to \R\times \Sigma^{n-1}$
by
$$
\psi_{s}(x):=\left(\frac{1}{\lambda_{i_s}(x)},x\right)\in 
\R\times \Sigma^{n-1}.
$$
Then
$
S:=\bigcup_{s=1}^r \psi_{s}(\Sigma^{n-1})
$
coincides with the singular set of $F$.
Since $F$ is complete null wave front, 
$S$ is compact. 
Since the projection of $S$ via the
continuous map 
$$
\R\times \Sigma^{n-1} \ni (t,x) \mapsto x\in \Sigma^{n-1}
$$
is compact, the hypersurface $\Sigma^{n-1}$ is
also compact.
Thus $r=n-1$ (i.e. $l=0$) and each $\lambda_i$
($i=1,\ldots,n-1$) is either positive-valued or
negative-valued on $\Sigma^{n-1}$.
Moreover, since $\Sigma^{n-1}$ is compact,
there exists a point $y_0\in \Sigma^{n-1}$
such that $f(y_0)$ attains
the farthest point of the image of $f$ 
from the origin in $\R^n_0$.
Then $\lambda_1(y_0),\,\, \dots,\,\, \lambda_{n-1}(y_0)$
are all positive or all negative at the same time.

Here, replacing $\nu$ by $-\nu$, we may assume that
$$
\lambda_{n-1}(x)>0 \qquad (x\in \Sigma^{n-1}).
$$
Then, $\lambda_1, \dots, \lambda_{n-1}$
take the same sign on $\Sigma^{n-1}$.
Thus, Hadamard's theorem \cite{H} implies that
$f$ is an embedded convex hypersurface in $\R^n_0$ whenever $n\ge 3$.
\end{proof}

If $n=2$, 
then $f$ is a locally strictly convex regular curve in $\R^2_0$.
So we may express $f$ by
$\gamma(s)$ $(s\in \R)$ 
defined as an $l$-periodic curve ($l>0$), 
parametrized by the arc-length.
Then its curvature function can be taken 
so that $\kappa(s)$ is positive everywhere.
Then we may assume that $F$ is 
expressed as
$$
F(t,s)=(t, \gamma_t(s))\qquad (s, t\in \R),
$$
where
$
\gamma_t(s):=\gamma(s)+t \nu(s)
$
is a parallel curve of $\gamma$,
and the singular set of $F$ is
$$
S:=\{(1/\kappa(s),s)\,;\, s\in \R \},
$$
and
\begin{equation}\label{eq:Cs}
C_\gamma(s):=F\left(\frac1{\kappa(s)},s\right)
\end{equation}
just parametrizes the singular set image of $F$.
We can prove the following:

\begin{proposition}\label{thm:V}
Let $F$ be a complete null wave front in $\R^3_1$
which is generated by a closed locally convex regular
curve $\gamma$ in $\R^2_0$.
Then non-cuspidal edge points of $F$
correspond to vertices on $\gamma$,
where a {\em vertex} is a critical point of
the curvature function of $\gamma$.
\end{proposition}

\begin{proof}
In the setting of \eqref{eq:Cs},
$C'_\gamma(s)\ne \mb 0$
if and only if
$\kappa'(s_0)\ne 0$, and so
$F$ has a cuspidal  edge singular point at
$(1/\kappa(s_0),s_0)$ 
only when $\kappa'(s_0)\ne 0$
(see the criterion in \cite{KRSUY} for cuspidal edge).
\end{proof}

The following assertion is equivalent to
the classical four vertex theorem for convex plane curve:

\begin{corollary}\label{cor:V}
Let $F$ be a complete null wave front in $\R^3_1$
which has no self-intersections outside of
a compact subset of $\R^3_1$. 
Then $F$ has at least four non-cuspidal edge 
singular points.
\end{corollary}

\begin{proof}
Since
$F$ has  no self-intersections outside of a compact
set, $F$ must be generated by a (closed) convex plane curve 
$\gamma$ in the $xy$-plane.
By the classical four vertex theorem,
$\gamma$ has at least four vertices, then such points
corresponds to the non-cuspidal edge points of $F$
as seen in the proof of Proposition~\ref{thm:V}.
So we obtain the assertion.
\end{proof}

Since a complete null wave front in $\R^3_1$
can be considered as a complete flat front in the Euclidean
$3$-space, Corollary \ref{cor:V} is a special case of
the four non-cuspidal edge point theorem given in
\cite[Theorem D]{MU}.

\begin{example}
We consider the complete null wave front  $F_a$ 
associated with the ellipse as in the introduction
which has four vertices when $0<a<1$.
They correspond to the swallowtail singular points of 
the complete null wave front in Figure \ref{fig:F0}
in the introduction, 
which satisfies the assumption of Corollary \ref{cor:V}.
\end{example}

\begin{figure}[htb]
 \begin{center}
\includegraphics[height=3.6cm]{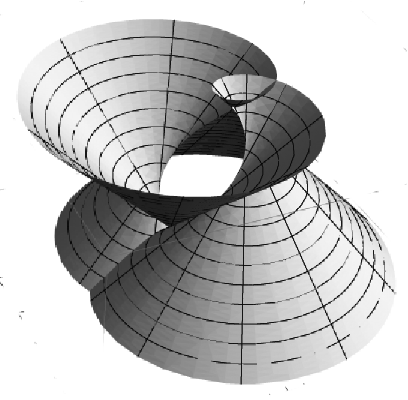}

\caption{The null wave front associated with 
a locally convex curve with exactly one crossing.
}
\label{F1}
\end{center}
\end{figure}

\begin{example}
Consider a locally convex plane curve with a self-intersection
$$
\gamma(\theta):=(1-2\sin \theta)(\cos \theta,\sin \theta)
\qquad (0\le \theta\le 2\pi),
$$
which admits only two vertices. These vertices 
correspond to two swallowtail singular points of 
the corresponding null wave front as in
Figure \ref{F1}.
\end{example}

\appendix
\section{A lemma of subspaces in $\R^{n+1}_1$}

We denote by $\inner{}{}$ the canonical
Lorentzian inner product of $\R^{n+1}_1$.
In this section, we prepare the following assertion for vector subspaces
in $\R^{n+1}_1$.

\begin{definition}
A subspace $V$ of $\R^{n+1}_1$ is said to be
{\it perpendicular} to a subspace $W$ if 
$
\inner{\mb v}{\mb w}=0\,\, (\mb v\in V,\,\, \mb w\in W)
$
holds.
\end{definition}

For a subspace $V$ of $\R^{n+1}_1$, 
we set
\begin{equation}\label{eq:V-perp}
V^\perp:=\{\mb x\in \R^{n+1}_1\,;\, \inner{\mb x}{\mb v}=0\,\, \text{for all $\mb v\in V$}\},
\end{equation}
which is the maximal subspace perpendicular to $V$ in $\R^{n+1}_1$.

\begin{definition}
A subspace 
$V$ of $\R^{n+1}_1$ is said to be {\it degenerate}
if there exists a non-zero vector $\mb v\in V$
such that
$
\inner{\mb v}{\mb w}=0
$
for all $\mb w\in V$. We call such a vector $\mb v$ a {\it degenerate vector} of $V$.
On the other hand,
a subspace
$V(\subset \R^{n+1}_1)$ is said to be  {\it non-degenerate}
if it is not degenerate.
\end{definition}

\begin{lemma}\label{prop:VWL}
Let $V,W$ and $N$ be three subspaces of $\R^{n+1}_1$
such that
\begin{enumerate}
\item $N$ is a 1-dimensional degenerate subspace
satisfying $N\cap W=\{\mb 0\}$,
\item $V$ and $W$ are perpendicular to $N$, and
\item $W$ is perpendicular to $V$.
\end{enumerate}
Then $V\cap W=\{\mb 0\}$.
\end{lemma}

\begin{proof}
If $W$ is non-degenerate, then so is $W^\perp$,
and using (3), we have
$$
\{\mb 0\}=W^\perp \cap W\supset V \cap W,
$$
which implies $V\cap W=\{\mb 0\}$.
So we may assume that $W$ is degenerate. 
By definition,
there exists a degenerate vector $\mb v_0$
belonging to $W$. 
We let $\xi\in N\setminus \{\mb 0\}$ be a generator of 
the 1-dimensional vector $N$.
Then $\mb v_0$ and $\xi$ are both degenerate vectors.
In particular, they are light-like, that is,
$\inner{\mb v_0}{\mb v_0}=\inner{\xi}{\xi}=0$ hold.
So if we set $\tilde W:=W+N$, then 
by (1), $\tilde W$ contains two
linearly independent light-like vectors
$\mb v_0$ and $\xi$, and so,
$\tilde W$ is a time-like vector space
(cf. \cite[Lemma~27, Page 141]{ON}),
which is non-degenerate.
So, $\tilde W^\perp$ is also non-degenerate.
Since $V\subset \tilde W^\perp$
(cf. (2) and (3))
and $W\subset \tilde W$,
we have
$$
\{\mb 0\}=\tilde W^\perp \cap \tilde W \supset V \cap W,
$$
proving $V\cap W=\{\mb 0\}$.
\end{proof}

\section{A method to make an immersion defined 
on a connected manifold 
from a family of embeddings.}

\begin{definition}\label{def:UV1601}
We fix positive integers $n,m$ ($n<m$).
Let $U,V$ be two connected $n$-dimensional manifolds. We
let $f:U\to \R^{m}$ and $g:V\to \R^{m}$ be
two $C^r$-embeddings.
A point $x\in U$ is said to be
{\it $(f,g)$-related} to $y\in V$
if there exist an open neighborhood
$O_x(\subset U)$ of $x$
and an open neighborhood
$O_y(\subset V)$ of $y$
such that
\begin{itemize}
\item  $f(x)=g(y)$ and
\item  $f(O_x)=g(O_y)$.
\end{itemize}
If there are no $(f,g)$-related points on $U$ and $V$, we say that
``{\it $U$ is not $(f,g)$-related to $V$}".
\end{definition}

\begin{remark}\label{rmk:1621}
This ``$(f,g)$-relatedness" is an open condition.
In fact, if $x\in U$ is 
$(f,g)$-related to $y\in V$,
then there exist open neighborhoods $O_x(\subset U)$ 
and $O_y(\subset V)$
of $x$ and $y$ respectively such that
each point of $O_x$ is
 $(f,g)$-related to a certain point in $O_y$.
\end{remark}

In this setting, the following assertion holds.

\begin{lemma}\label{prop:unique}
If we set
\begin{align*}
O_{U,V}&:=\{x\in U\,;\, \text{$x$ is $(f,g)$-related to some $y\in V$}\}, \\
O_{V,U}&:=\{y\in V\,;\, \text{$y$ is $(g,f)$-related to some $x\in U$}\},
\end{align*}
then $O_{U,V}$ $($resp.~$O_{V,U})$ is an open subset of $U$ $($resp.~$V)$.
Moreover, if $O_{U,V}$ is non-empty,
then there exists a unique $C^r$-diffeomorphism
$\phi:O_{U,V}\to O_{V,U}$ satisfying
$g\circ\phi=f$ on $O_{U,V}$.
\end{lemma}

\begin{proof}
Assume that
$x\in U$ (resp. $y\in V$)
is $(f,g)$-related to $y\in V$ (resp. $x\in U$).
Since $f$ and $g$ are embeddings,
 $y\in V$ (resp. $x\in U$)
is uniquely determined, and so
the map $\phi$ is also uniquely determined.
The smoothness of $\phi$ is obvious.
\end{proof}

\begin{definition}\label{def:1632}
Let $f:U\to \R^{m}$ and $g:V\to \R^{m}$ 
be as in
Definition \ref{def:UV1601}.
Then the pair $(U,V)$ is said to be
{\it $(f,g)$-admissible} if,
for each pair $(x,y)\in U\times V$,
it holds that
\begin{itemize}
\item[(i)] $x$ is  $(f,g)$-related to $y$, or
\item[(ii)] there exist a neighborhood $O_x(\subset U)$ of $x$
and a neighborhood $O_y(\subset V)$ of $y$
such that $O_x$
is not $(f,g)$-related to $O_y$.
\end{itemize}
\end{definition}

By definition, if $f(U)$ does not meet $g(V)$,
then the pair $(U,V)$ is $(f,g)$-admissible.
The following assertion is immediately from the definition:

\begin{lemma}\label{lem:1653}
In the setting of
Definition \ref{def:1632},
the pair $(U,V)$ is {\it $(f,g)$-admissible} 
if $f(U)$ meets $g(V)$ transversally or 
$f$ and $g$ are both real analytic.
\end{lemma}

\begin{proof}
We fix a pair $(x,y)\in U\times V$ such that $x$ is not $(f,g)$-related to $y$.
If $f(U)$ meets $g(V)$ transversally, then the assertion is obvious.
So we may assume that $f$ and $g$ are both real analytic and
$f(U)$ does not meet $g(V)$ transversally.
Then $P:=f(x)=f(y)$
and we can take $n$-dimensional affine plane $T^n$ 
as a common tangential space of $f(U)$ and $g(V)$ at $P$ in $\R^m$.
Since $U$ (resp. $V$) is locally connected, there exists a 
connected open neighborhood $U_1$ (resp. $V_1$) of $x$ (resp. $y$)
such that $U_1\subset U$ (resp. $V_1\subset V$).
By the implicit function theorem, we may assume that the images of the maps
$f|_{U_1}$ and $g|_{V_1}$ are expressed as the graphs of certain
functions $F,G:\Omega\to \R^{m-n}$ 
defined on the same domain $\Omega$ in $T^n$.

It suffices to show that there exist a neighborhood 
 $O_x(\subset U_1)$ of $x$
and a neighborhood $O_y(\subset V_1)$ of $y$ such that 
$O_x$ is not $(f,g)$-related to $O_y$.
If not, there exists a pair $(x_1,y_1)\in O_x\times O_y$
such that $x_1$ is $(f,g)$-related to $y_1$,
which implies that there exist open neighborhoods $O_1(\subset O_x)$ 
and $O_2(\subset O_y)$ of $x_1$ and $y_1$, respectively, such that
each point of $O_1$ is $(f,g)$-related to a corresponding point in $O_2$.
Then the vector-valued function $F$ coincides with $G$ on
some non-empty open subset of $\Omega$.
By the connectedness of $\Omega$ and 
the real analyticity of $F$ and $G$, the two
vector-valued functions coincide identically on $\Omega$,
which implies that $x$ is $(f,g)$-related to $y$, a contradiction.
\end{proof}

\begin{definition}\label{def:1668}
Let $\mc A$ be an at most countable set, and 
let 
$\{(U_\lambda,f_\lambda)\}_{\lambda\in \mc A}$ be 
a family consisting of connected
$n$-dimensional manifolds $U_\lambda$
and embeddings $f_\lambda:U_\lambda\to \R^m$.
Then  $\{(U_\lambda,f_\lambda)\}_{\lambda\in \mc A}$ 
is called {\it admissible}
if $(U_\lambda,U_\mu)$ is $(f_\lambda,f_\mu)$-admissible
for any choice of $(\lambda,\mu)\in \mc A\times \mc A$.
\end{definition}

We let
$\{(U_\lambda,f_\lambda)\}_{\lambda\in \mc A}$ 
be
an admissible family as in
Definition \ref{def:1668}.
Consider the disjoint union
$
\mathcal S:=\coprod_{\lambda\in \mc A} U_\lambda,
$
and define
a relation $x\sim y$ for 
$x,y\in \mathcal S$ so that $x\sim y$ implies that
there exist a neighborhood $U_\lambda$ ($\lambda\in \mc A$)
of $x$
and a
neighborhood $U_\mu$ ($\mu\in \mc A$)
of $y$
such that $x$ is $(f_\lambda,f_\mu)$-related to $y$.
Then it is easy to check that $\sim$ is an equivalence relation.
Moreover, the following assertion holds:

\begin{proposition}\label{prop:1671}
Let
$\{(U_\lambda,f_\lambda)\}_{\lambda\in \mc A}$ be an admissible family.
Then there exist
\begin{itemize}
\item a manifold $M^n:=\mathcal S/{\sim}$,
\item a $C^r$-immersion $g:M^n\to \R^m$, and
\item a diffeomorphism $\Psi_\lambda:U_\lambda\to \Psi_\lambda(U_\lambda)\subset M^n$
\end{itemize}
such that
$\{(U_\lambda,\Psi_\lambda)\}_{\lambda\in \mc A}$ is
a differentiable structure of $M^n$
and 
$g\circ \Psi_\lambda$ coincides with 
$f_\lambda$ on $U_\lambda$ for each $\lambda\in \mc A$.
\end{proposition}

\begin{proof}
Since $\mathcal S$ is a disjoint union of open subsets, 
this $U_\lambda$ is uniquely determined by $x$.
So we denote $U_\lambda$ by $V_x$.
As we have already noted,
$\sim$ is an equivalence relation, and
the canonical projection
$
\pi:\mathcal S\to M^n:=\mathcal S/{\sim}
$
is induced.
Since $\mc A$ is an at most countable set,
$M^n$ satisfies the second axiom of countability.
If we show that the quotient space
$M^n$ is a Hausdorff space,
then we can easy observe that $M^n$
has a structure of $C^r$-manifold by using
Lemma~\ref{prop:unique},
and we can 
construct
the desired
$C^r$-immersion $g:M^n\to \R^m$
by setting 
\begin{equation}\label{eq:2542}
\Psi_\lambda:=\pi|_{U_\lambda}.
\end{equation}
So it is sufficient to prove that $M^n$ is a Hausdorff space:
Consider the set defined by
$$
\mathcal R:=\{(p,q)\in \mathcal S\times \mathcal S\,;\, p\,\,\sim\,\, q\}.
$$
By Remark \ref{rmk:1621},
the canonical projection
$\pi:\mathcal S\to M^n$ is an open map.
So we prove that
$\mathcal R$ is a closed subset of 
$\mathcal S\times \mathcal S$
(cf. \cite[Chapter 3, Theorem~11]{K0}).
We consider a sequence
$
\{(p_k,q_k)\}_{k=1}^\infty
$
in $\mathcal R$ and suppose that $\{p_k\}_{k=1}^\infty$ and $\{q_k\}_{k=1}^\infty$
converge to $p$ and $q$ in $\mathcal S$, respectively. 
By definition, there exist  $(U_\lambda,f_\lambda)$
and $(U_\mu,f_\mu)$ ($\lambda,\mu\in \mc A$)
and a positive integer $l$
such that 
\begin{itemize}
\item $V_p=U_\lambda$ and $V_q=U_\mu$,
\item $V_{p_k}=U_\lambda$ and $V_{q_k}=U_\mu$ for $k\ge l$.
\end{itemize}
We remark that the admissibility of $\{(U_\lambda,f_\lambda)\}_{\lambda
\in \mc A}$ implies 
the admissibility between $U_\lambda$ and $U_\mu$.
Since $p_k\to p$ and $q_k\to q$, we have
$f_\lambda(p)=f_\mu(q)$.
We suppose $p\not\sim q$. Then,
by (ii) of Definition 
\ref{def:1632}, there exist
a neighborhood $O_p(\subset U_\lambda)$ of $p$
and a neighborhood $O_q(\subset U_\mu)$ of 
$q$ such that 
$O_p$
is not $(f_\lambda,f_\mu)$-related to $O_q$.
However, this contradicts the fact that $p_k\sim q_k$.
So $M^n$ is a Hausdorff space.
\end{proof}

\begin{acknowledgements}
The authors thank Riku Kishida and the reviewer for valuable comments.
\end{acknowledgements}

\end{document}